\documentclass{amsart}

\usepackage{graphicx}
\usepackage[shortlabels]{enumitem}

\usepackage{color}

\usepackage{tikz}
\usetikzlibrary{arrows,chains,matrix,positioning,scopes}
\usetikzlibrary{arrows.meta}
\usetikzlibrary{matrix}

\usepackage{hyperref}
\usepackage {cleveref}

\newtheorem{theorem}{Theorem}[section]
\newtheorem{proposition}[theorem]{Proposition}
\newtheorem{lemma}[theorem]{Lemma}
\newtheorem{corollary}[theorem]{Corollary}

\theoremstyle{definition}
\newtheorem{definition}[theorem]{Definition}

\theoremstyle{remark}
\newtheorem{remark}[theorem]{Remark}

\numberwithin{equation}{section}

\begin{document}

\title{Cusp cobordism group of Morse functions}

\author[D.J. Wrazidlo]{Dominik J. Wrazidlo}

\address{Institute of Mathematics for Industry, Kyushu University, Motooka 744, Nishi-ku, Fukuoka 819-0395, Japan}
\email{d-wrazidlo@imi.kyushu-u.ac.jp}


\date{\today.}

\begin{abstract}
By a Morse function on a compact manifold with boundary we mean a real-valued function without critical points near the boundary such that its critical points as well as the critical points of its restriction to the boundary are all non-degenerate.
For such Morse functions, Saeki and Yamamoto have previously defined a certain notion of cusp cobordism, and computed the unoriented cusp cobordism group of Morse functions on surfaces.
In this paper, we compute unoriented and oriented cusp cobordism groups of Morse functions on manifolds of any dimension by employing Levine's cusp elimination technique as well as the complementary process of creating pairs of cusps along fold lines.
We show that both groups are cyclic of order two in even dimensions, and cyclic of infinite order in odd dimensions.
For Morse functions on surfaces our result yields an explicit description of Saeki-Yamamoto's cobordism invariant which they constructed by means of the cohomology of the universal complex of singular fibers.
\end{abstract}

\maketitle

\section{Introduction}

In differential topology, cobordism groups of maps with prescribed singularity type can generally been studied by means of stable homotopy theory and related methods of algebraic topology.
The first fundamental result in the topic is due to Ren\'{e} Thom \cite{thom}, who applied the Pontrjagin-Thom construction to study embedded submanifolds in a Euclidean space up to cobordism.
An adaption of Thom's approach to immersions of manifolds was given by Wells \cite{wel}.
Later, Rim\'{a}nyi and Sz\H{u}cs \cite{rim} used the concept of $\tau$-maps to develop a Pontrjagin-Thom type construction in order to study
cobordism of maps of positive codimension with given types of stable singular map germs.
In the sequel, their results have been extended vastly by several authors including Ando \cite{ando}, Kalm\`{a}r \cite{kal1}, Sadykov \cite{sad2}, and Sz\H{u}cs \cite{szu}.

In the special case of maps with target dimension one, the $C^{\infty}$ stable singular map germs $(\mathbb{R}^{n}, 0) \rightarrow (\mathbb{R}, 0)$ are precisely non-degenerate critical points, and cobordism theory for such maps is based on $C^{\infty}$ stable singular map germs $(\mathbb{R}^{n+1}, 0) \rightarrow (\mathbb{R}^{2}, 0)$, namely fold points and cusps.
Several authors have studied various notions of cobordism of Morse functions by invoking geometric-topological methods.
Ikegami \cite{ike} uses Levine's cusp elimination technique \cite{lev} to compute cobordism groups of Morse functions on closed manifolds.
His work generalizes previous results of Ikegami-Saeki \cite{ikesae} for Morse functions on oriented surfaces, and of Kalm\`{a}r \cite{kal} for Morse functions on unoriented surfaces.
Later, Ikegami-Saeki \cite{is} modified Ikegami's approach in order to study cobordism groups of Morse maps into the circle.
Using the technique of Stein factorization, Saeki \cite{sae} showed that cobordism groups of Morse functions with only maxima and minima as their singularities are isomorphic to the group of homotopy spheres.
In \cite{wra3}, the author considers cobordism of Morse functions subject to more general index constraints, and shows that exotic Kervaire spheres can be distinguished from other exotic spheres as elements of such cobordism groups in infinitely many dimensions.

Another direction of research has been initiated in \cite{sy1}, where Saeki and Yamamoto study cobordism groups of Morse functions on compact manifolds possibly with boundary.
More precisely, they consider $C^{\infty}$ stable real-valued functions up to so-called \emph{admissible cobordism}, a notion which is based on proper $C^{\infty}$ stable maps of manifolds possibly with boundary into the plane that are submersions near the boundary.
For Morse functions on compact unoriented surfaces possibly with boundary, Saeki and Yamamoto derive a cobordism invariant with values in the cyclic group of order two.
Their method is based on an examination of the cohomology of the universal complex of singular fibers \cite{sae2, sae3}.
Furthermore, in \cite{sy3} they present a combinatorial argument using labeled Reeb graphs to show that no further non-trivial invariants exist.
Consequently, the admissible cobordism group of such Morse functions is in fact isomorphic to the cyclic group of order two, which answers a conjecture posed in \cite{sy2}.
Recently, cusp-free versions of cobordism for Morse functions on compact unoriented surfaces with boundary have been studied by Yamamoto \cite{yam} by means of similar techniques.
\par\medskip

In this paper, we introduce the notion of \emph{cusp cobordism} for Morse functions on compact manifolds possibly with boundary (see \Cref{definition cusp cobordism}), and compute the unoriented and oriented versions of the associated cobordism group in dimension $\geq 2$.
We show that both groups are cyclic of order two in even dimensions, and cyclic of infinite order in odd dimensions (see \Cref{main theorem} and \Cref{remark orientation}).
The notion of cusp cobordism is based on maps into the plane that are generic (see \Cref{definition generic maps into the plane}), which means that the critical point set consists of only fold points and cusps.
Roughly speaking, the distinction between even and odd dimensions in our result is related to the phenomenon that circle-shaped components of the singular set of a generic map can contain an odd number of cusps if and only if the dimension of the source manifold is even.
In order to control the parity of the number of cusps on the components of the singular set we combine Levine's cusp elimination technique (see \cite{lev} and \Cref{Elimination of cusps}) with the complementary process of creating pairs of cusps along fold lines (see \Cref{creation of cusps}).
Moreover, for generic maps between surfaces, we use the method of singular patterns developed by the author in \cite{wra2}.

Note that our notion of cusp cobordism differs slightly from that of admissible cobordism in that we use generic maps into the plane without making $C^{\infty}$ stability assumptions on the maps.
Nevertheless, we show in \Cref{proposition admissible cobordism group} that both notions give rise to isomorphic cobordism groups.
Thus, for Morse functions on unoriented surfaces we obtain an explicit description of Saeki-Yamamoto's cobordism invariant of \cite{sy1}.
What is more, our results answer Problem 6.1 and Problem 6.2 that were posed by Saeki and Yamamoto in \cite{sy3} concerning the computation of admissible cobordism groups in higher dimensions.

The paper is structured as follows.
In \Cref{Statement of result} we state our main results in detail.
In \Cref{Preliminaries on generic maps into the plane} we provide some background from singularity theory of generic maps into the plane.
\Cref{Non-singular extensions} and \Cref{Creation and elimination of cusps} are concerned with several technical results that will be used in the proof of our main result.
The proof of \Cref{main theorem} is given in \Cref{proof of main theorem}.
Finally, \Cref{admissible cobordism group} relates our notion of cusp cobordism to Saeki-Yamamoto's notion of admissible cobordism.

Unless otherwise stated, all manifolds (possibly with boundary) and maps between manifolds are assumed to be smooth, that is, differentiable of class $C^{\infty}$.
Given a map $f \colon M \rightarrow N$ between manifolds, the set of singular points of $f$ will be denoted by $S(f)$.
For an oriented manifold $M$, the manifold with opposite orientation will be denoted by $-M$.
We denote the cardinality of a set $X$ by $\# X$.

\subsection*{Acknowledgements}
The author would like to express his gratitude to Professor Osamu Saeki for stimulating discussions.

This work was done while the author was an International Research Fellow of Japan Society for the Promotion of Science (Postdoctoral Fellowships for Research in Japan (Standard)).

\section{Statement of main result}\label{Statement of result}

Given a real-valued function $h \colon U \rightarrow \mathbb{R}$ defined on an $m$-manifold without boundary $U$, a critical point $x \in S(h)$ is non-degenerate if there is a chart of $U$ centered at $x$ in which $h$ takes the form
$$
(x_{1}, \dots, x_{m}) \mapsto h(x) - x_{1}^{2} - \dots - x_{i}^{2} + x_{i+1}^{2} + \dots + x_{m}^{2}.
$$
We also recall that the number $i$ of minus signs appearing in the above standard quadratic form does not depend on the choice of the chart centered at $x$, and is called the (Morse) index of $h$ at $x$.
In the following, we will denote the set of non-degenerate critical points of $h$ of Morse index $i$ by $S^{i}(h)$.

Let $n \geq 2$ be an integer.
Throughout this paper, by a Morse function on a compact $n$-manifold possibly with boundary $M^{n}$ (see \Cref{figure morse example}) we mean a real-valued function $f \colon M \rightarrow \mathbb{R}$ which is a submersion at every point of $\partial M$, and such that the critical points of both $f$ and $f|_{\partial M}$ are all non-degenerate.
Note the difference of our setting to others, where Morse functions are allowed to have critical points on the boundary (for example, see Definition 1.4 in \cite{bnr}).
Our notion of Morse functions on $M^{n}$ arises naturally in the context of stable mappings.
Namely, as pointed out in the introduction of \cite{sy3}, a function $f \colon M \rightarrow \mathbb{R}$ is known to be $C^{\infty}$ stable if and only if $f$ is a Morse function that is injective on $S(f) \sqcup S(f|_{\partial M})$.

\begin{figure}[htbp]
\centering
\fbox{\begin{tikzpicture}
    \draw (0, 0) node {\includegraphics[height=0.4\textwidth]{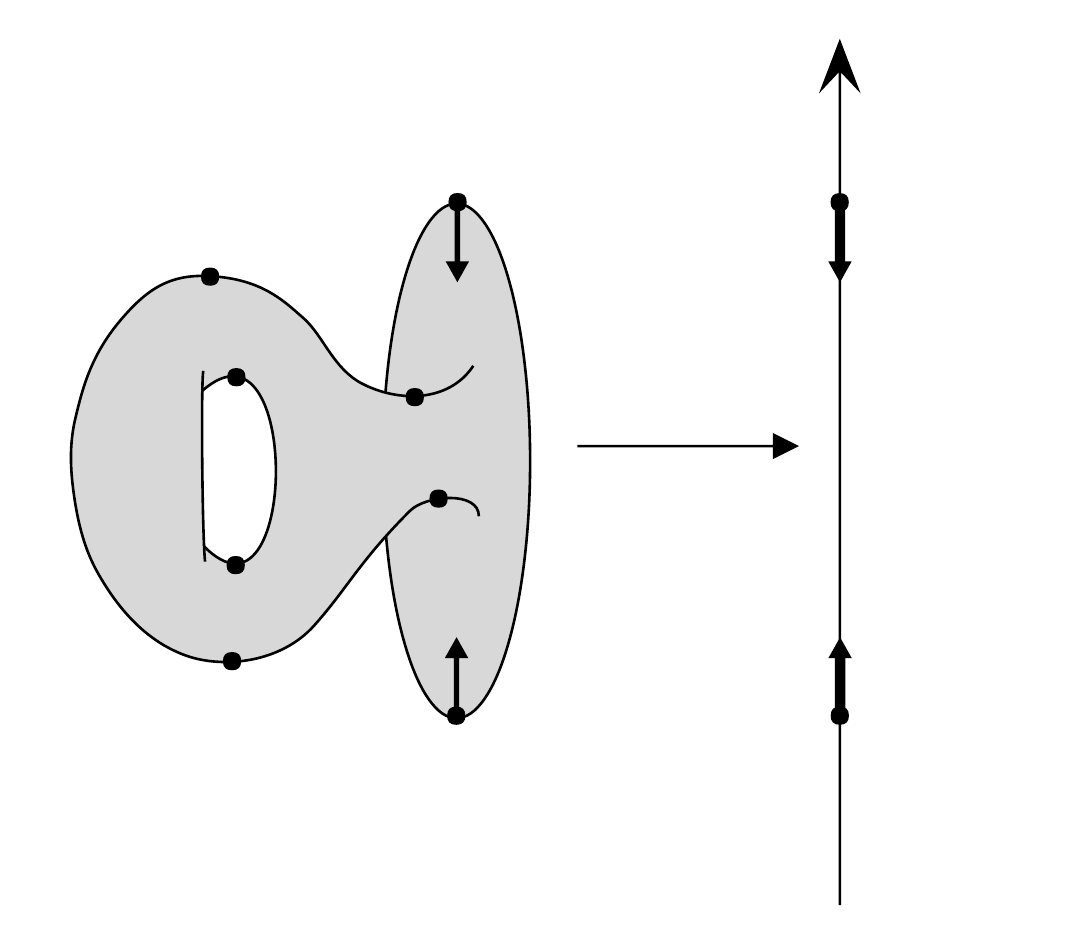}};
    \draw (2.1, 2.2) node {$\mathbb{R}$};
    \draw (2.5, 1.2) node {$df_{x_{1}}(v_{1})$};
    \draw (-0.35, -1.55) node {$x_{0}$};
    \draw (0.8, 0.4) node {$f$};
    \draw (-1.9, -1.6) node {$M^{2}$};
    \draw (-0.35, 1.65) node {$x_{1}$};
    \draw (2.5, -1.1) node {$df_{x_{0}}(v_{0})$};
\end{tikzpicture}}
\caption{Illustration of a Morse function $f \colon M \rightarrow \mathbb{R}$ on a compact surface with boundary embedded in $\mathbb{R}^{3}$ induced by the height function.
The critical points of $f|_{\partial M}$ are $x_{0}$ and $x_{1}$.
Using inward pointing tangent vectors $v_{0} \in T_{x_{0}}M$ and $v_{1} \in T_{x_{1}}M$ as indicated, we see that $\sigma_{f}(x_{0}) = +1$ and $\sigma_{f}(x_{1}) = -1$.
Hence, $S_{+}[f] = S_{+}^{0}[f] = \{x_{0}\}$, and $\chi_{+}[f] = 1$.}
\label{figure morse example}
\end{figure}
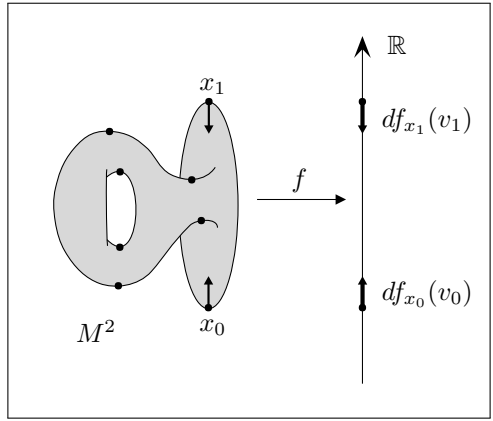

In this paper, we are concerned with cobordism theory of Morse functions on compact manifolds possibly with boundary.
Our notion of cusp cobordism (see \Cref{definition cusp cobordism} below) uses the following well-known cobordism relation for pairs $(M_{0}, M_{1})$ of compact oriented $n$-manifolds possibly with boundary.
An oriented cobordism $(W, V)$ from $M_{0}$ to $M_{1}$ (compare \Cref{figure injectivity}) is a compact oriented $(n+1)$-manifold with corners $W$ such that there is a decomposition $\partial W = M_{0} \cup_{\partial M_{0}} V \cup_{- \partial M_{1}} - M_{1}$, where $M_{0}$, $-M_{1}$ and $V$ are oriented codimension $0$ submanifolds of $\partial W$ satisfying $M_{0} \cap M_{1} = \emptyset$ as well as $V \cap M_{0} = \partial M_{0}$ and $V \cap M_{1} = \partial M_{1}$, $V^{n}$ is an oriented cobordism from $\partial M_{0}$ to $\partial M_{1}$, and $W$ has corners precisely along $\partial V$.
Any two compact oriented $n$-manifolds possibly with boundary can be seen to be oriented cobordant.
However, it might not be possible to extend given Morse functions on them to a cusp cobordism in the sense of the following definition.

\begin{definition}[cusp cobordism]\label{definition cusp cobordism}
For $i = 0, 1$ let $f_{i} \colon M_{i} \rightarrow \mathbb{R}$ be a Morse function defined on a compact oriented $n$-manifold possibly with boundary $M_{i}$.
A \emph{cusp cobordism} from $f_{0}$ to $f_{1}$ is an oriented cobordism $(W, V)$ from $M_{0}$ to $M_{1}$ together with a map $F \colon W \rightarrow [0, 1] \times \mathbb{R}$ such that $F^{-1}(\mathbb{R} \times \{i\}) = M_{i}$ for $i = 0, 1$, and the following properties hold:
\begin{enumerate}[(i)]
\item For some $\varepsilon > 0$ there exist collar neighborhoods (with corners) $[0, \varepsilon) \times M_{0} \subset W$ of $\{0\} \times M_{0} = M_{0} \subset W$ and $(1-\varepsilon, 1] \times M_{1} \subset W$ of $\{1\} \times M_{1} = M_{1} \subset W$ such that $F|_{[0, \varepsilon) \times M_{0}} = \operatorname{id}_{[0, \varepsilon)} \times f_{0}$ and $F|_{(1-\varepsilon, 1] \times M_{1}} = \operatorname{id}_{(1-\varepsilon, 1]} \times f_{1}$.
\item The restriction $F|_{W \setminus (M_{0} \sqcup M_{1})}$ is a submersion at every point of the boundary $V \setminus \partial V$ of $W \setminus (M_{0} \sqcup M_{1})$.
\item The restrictions $F|_{W \setminus \partial W}$ and $F|_{V \setminus \partial V}$ are generic maps into the plane, that is, their singular sets consist of fold points and cusps (see \Cref{definition generic maps into the plane}).
\end{enumerate}
\end{definition}

\begin{remark}\label{remark target plane}
In the formulation of \Cref{definition cusp cobordism} we may equivalently use maps $W \rightarrow \mathbb{R}^{2}$ instead of maps $F \colon W \rightarrow [0, 1] \times \mathbb{R}$ such that $F^{-1}(\mathbb{R} \times \{i\}) = M_{i}$ for $i = 0, 1$.
In fact, any map $F_{0} \colon W \rightarrow \mathbb{R}^{2}$ with the properties (i) to (iii) of \Cref{definition cusp cobordism} can be modified to a map $F \colon W \rightarrow [0, 1] \times \mathbb{R}$ such that $F^{-1}(\mathbb{R} \times \{i\}) = M_{i}$ for $i = 0, 1$ as follows.
Since $W$ is compact, we may consider $F_{0}$ for suitable $a < 0$ and $b > 1$ as a map $F_{0} \colon W \rightarrow [a, b] \times \mathbb{R}$ such that $F_{0}(W) \subset (a, b) \times \mathbb{R}$.
Next, we choose a diffeomorphism $\xi \colon [0, 1] \stackrel{\cong}{\longrightarrow} [a, b]$ that extends the identity map on $(\varepsilon/2, 1-\varepsilon/2)$ (where $\varepsilon > 0$ is taken from property (i) of $F_{0}$).
Then, we modify $F_{0}$ near $\partial W$ by defining the map $F_{1} \colon W \rightarrow [a, b] \times \mathbb{R}$ via $F_{1}|_{[0, \varepsilon) \times M_{0}} = \xi|_{[0, \varepsilon)} \times f_{0}$, $F_{1}|_{(1-\varepsilon, 1] \times M_{1}} = \xi|_{(1-\varepsilon, 1]} \times f_{1}$, and $F_{1}(x) = F_{0}(x)$ for all $x \in W$ with $x \notin ([0, \varepsilon/2] \times M_{0}) \sqcup ([1-\varepsilon/2, 1] \times M_{1})$.
Finally, the composition $F = (\xi^{-1} \times \operatorname{id}_{\mathbb{R}}) \circ F_{1} \colon W \rightarrow [0, 1] \times \mathbb{R}$ has the desired properties.
Note that the properties (i) to (iii) of \Cref{definition cusp cobordism} are not affected by our modifications.
\end{remark}

The relation of cusp cobordism defined above is an equivalence relation on the set of Morse functions on compact oriented $n$-manifolds possibly with boundary.
Moreover, the set $\mathcal{C}_{n}$ of equivalence classes has a natural group structure induced by disjoint union.
Note that the identity element is represented by the unique function $f_{\emptyset} \colon \emptyset \rightarrow \mathbb{R}$ on the empty set, and the inverse of a class $[f \colon M \rightarrow \mathbb{R}]$ is represented by the function $-f \colon - M \rightarrow \mathbb{R}$, $(-f)(x) = -f(x)$, where we recall that $-M$ denotes the manifold $M$ with reversed orientation.
The purpose of this paper is to determine the structure of the group $\mathcal{C}_{n}$ for $n \geq 2$ (see our main result, \Cref{main theorem}).
For lower dimensional versions of $\mathcal{C}_{n}$, see \Cref{remark lower dimension}.

Before stating \Cref{main theorem}, we need to introduce some notation for functions defined on manifolds possibly with boundary.
First, we recall that the Euler characteristic $\chi(K)$ of a finite CW complex $K$ can be defined as the number of even dimensional cells minus the number of odd dimensional cells of $K$.
The Euler characteristic is known to depend only on the homotopy type of $K$.
If $h \colon P \rightarrow \mathbb{R}$ is a Morse function on a closed $(n-1)$-manifold $P$, then by classical Morse theory \cite{mil2}, $P$ is homotopy equivalent to a finite CW complex of dimension $n-1$ whose $i$-cells correspond to the critical points in $S^{i}(h)$.
Hence, the Euler characteristic of $P$ is given by
\begin{align}\label{euler characteristic}
\chi(P) = \sum_{i=0}^{n-1} (-1)^{i} \cdot \# S^{i}(h).
\end{align}
Next, we consider a real-valued function $g \colon N^{n} \rightarrow \mathbb{R}$ defined on some $n$-manifold possibly with boundary.
Let us suppose that $g$ is a submersion in a neighborhood of the boundary, and that $g$ restricts to a Morse function $\partial N \rightarrow \mathbb{R}$.
Then, we can assign to every critical point $x$ of the Morse function $g|_{\partial N}$ a sign $\sigma_{g}(x) \in \{\pm 1\}$ that is uniquely determined by requiring that for an \emph{inward} pointing tangent vector $v \in T_{x}N$ the tangent vector
$$
\sigma_{g}(x) \cdot dg_{x}(v) \in T_{g(x)} \mathbb{R} = \mathbb{R}
$$
points into the \emph{positive} direction of the real axis.
In fact, the resulting assignment
\begin{align}\label{sign function}
\sigma_{g} \colon S(g|_{\partial N}) \rightarrow \{\pm 1\}
\end{align}
depends only on the germ $[g]$ of $g$ near $\partial N$.
Let $S_{+}[g] \subset S(g|_{\partial N})$ denote the subset of those critical points $x$ of the Morse function $g|_{\partial N}$ for which $\sigma_{g}(x) = +1$, and let $S_{+}^{i}[g] = S^{i}(g|_{\partial N}) \cap S_{+}[g]$.
If $\partial N$ is compact, then the number of critical points of $g|_{\partial N}$ is finite, and in analogy with \Cref{euler characteristic} we define the integer
\begin{align}\label{signed Euler characteristic}
\chi_{+}[g] = \sum_{i=0}^{n-1} (-1)^{i} \cdot \# S_{+}^{i}[g].
\end{align}
In particular, every Morse function $f \colon M \rightarrow \mathbb{R}$ defined on a compact $n$-manifold possibly with boundary has an associated integer $\chi_{+}[f]$ (for example, compare \Cref{figure morse example} and \Cref{figure morse generator example}).

Our main result is the following

\begin{theorem}\label{main theorem}
Let $n \geq 2$ be an integer.
Assigning to every Morse function $f \colon M^{n} \rightarrow \mathbb{R}$ on a compact oriented $n$-manifold possibly with boundary the integer $\chi(M) - \chi_{+}[f]$ induces group isomorphisms
\begin{align*}
\mathcal{C}_{n} \stackrel{\cong}{\longrightarrow} \begin{cases}
\mathbb{Z}/2, \qquad &n \text{ even}, \\
\mathbb{Z}, \qquad &n \text{ odd}.
\end{cases}
\end{align*}
\end{theorem}

We finish this section with some remarks on our result.

\begin{remark}[orientations]\label{remark orientation}
It is straightforward to define the unoriented cusp cobordism group $\widetilde{\mathcal{C}}_{n}$ by forgetting about orientations of manifolds in the definition of $\mathcal{C}_{n}$.
As it turns out, the two groups $\mathcal{C}_{n}$ and $\widetilde{\mathcal{C}}_{n}$ are isomorphic because the arguments used in this paper to prove \Cref{main theorem} and \Cref{proposition admissible cobordism group} do not exploit orientations of manifolds anywhere.
Therefore, it suffices to focus on the group $\mathcal{C}_{n}$.
\end{remark}

\begin{remark}[comparison with admissible cobordism]
Our notion of cusp cobordism (see \Cref{definition cusp cobordism}) differs from the notion of admissible cobordism due to Saeki and Yamamoto (see \Cref{definition admissible cobordism}) in that we work with generic maps into the plane without imposing any $C^{\infty}$ stability requirements on the mappings.
The methods used in \cite{sy1, sy3} exploit the $C^{\infty}$ stability assumptions to compute the unoriented admissible cobordism group $b \mathfrak{N}_{2}$ of Morse functions on compact surfaces possibly with boundary.
We show in \Cref{proposition admissible cobordism group} that the relations of cusp cobordism and admissible cobordism result in isomorphic cobordism groups.
In particular, for Morse functions on unoriented surfaces our results yield an explicit description of the cobordism invariant $b \mathfrak{N}_{2} \rightarrow \mathbb{Z}/2$ which has been constructed by Saeki and Yamamoto in Corollary 4.9(1) of \cite{sy1}.
Namely, Saeki and Yamamoto identify the Morse function shown in \Cref{figure morse generator example} as a generator of $b \mathfrak{N}_{2}$.
This is in fact consistent with evaluation under the isomorphisms provided by \Cref{main theorem} and \Cref{proposition admissible cobordism group}.
What is more, our results yield the computation of (un-)oriented admissible cobordism groups in arbitrary dimension, thus answering Problem 6.1 and Problem 6.2 (compare \Cref{remark orientation}) in \cite{sy3}.
\end{remark}

\begin{remark}[lower dimensions]\label{remark lower dimension}
In Section 5 of \cite{sy3}, Saeki and Yamamoto define admissible cobordism relations for Morse functions on manifolds of dimension $0$ and $1$, and show that the resulting unoriented cobordism groups are explicitly given by $b \mathfrak{N}_{0} = 0$ and $b \mathfrak{N}_{1} \cong \mathbb{Z}$.
It is clear how to adapt the results of Section 5 of \cite{sy3} to the setting of oriented manifolds, which allows us to define oriented admissible cobordism groups $b \mathfrak{M}_{0}$ and $b \mathfrak{M}_{1}$ (compare Problem 6.2 in \cite{sy3}), and to show by essentially the same proofs that $b \mathfrak{M}_{0} = 0$ and $b \mathfrak{M}_{1} \cong \mathbb{Z}$.
For $i = 0, 1$ we define the oriented and unoriented cusp cobordism groups $\mathcal{C}_{i}$ and $\widetilde{\mathcal{C}}_{i}$ by omitting the $C^{\infty}$ stability assumptions made in the definitions of $b \mathfrak{M}_{i}$ and $b \mathfrak{N}_{i}$, respectively.
As the $C^{\infty}$ stability assumptions are included only for formal reasons in Section 5 of \cite{sy3} without affecting the proofs, we conclude that $b \mathfrak{M}_{i} \cong \mathcal{C}_{i}$ and $b \mathfrak{N}_{i} \cong \widetilde{\mathcal{C}}_{i}$ for $i = 0, 1$.
\end{remark}

\begin{remark}[non-singular Morse functions]\label{remark non-singular Morse functions}
Let us discuss a connection of our cobordism viewpoint with the classical problem of extending a given Morse function $\partial M \rightarrow \mathbb{R}$ defined on the boundary of a compact orientable $n$-manifold $M$ to a Morse function $M \rightarrow \mathbb{R}$ without critical points.
In the 1970s, the cases $n=2$ and $n=3$ of the problem were studied by Blank-Laudenbach \cite{bl} and Curley \cite{curl}, respectively.
Later, Barannikov \cite{bar} and Seigneur \cite{sei} derived necessary algebraic conditions for the existence of an extension in terms of the Morse complex when $n$ is arbitrary and $M^{n} = S^{n}$.
It can be shown directly that every Morse function $f \colon M \rightarrow \mathbb{R}$ without critical points represents the trivial element of $\mathcal{C}_{n}$.
Namely, a nullcobordism $F \colon W \rightarrow \mathbb{R}^{2}$ of $f$ can be obtained by restricting the map $\operatorname{id}_{[0, 1]} \times f \colon [0, 1] \times M \rightarrow \mathbb{R}^{2}$ to a codimension $0$ submanifold $W \subset [0, 1] \times M$ with boundary $\partial W = (\{0\} \times M) \cup_{\partial M} \iota(M)$ and corners along $\partial (\{0\} \times M) = \partial M = \partial (\iota(M))$, where $\iota \colon M \rightarrow [0, 1] \times M$ is an embedding such that $\iota(t, x) = (t, x)$ on some collar $[0, \varepsilon) \times \partial M \subset M$, and the composition $(\operatorname{id}_{[0, 1]} \times f) \circ \iota|_{M \setminus \partial M}$ is generic.
Consequently, as a necessary condition for the extendability of a non-singular $\partial M$-germ $[g \colon [0, \varepsilon) \times \partial M \rightarrow \mathbb{R}]$ to a non-singular Morse function $M \rightarrow \mathbb{R}$, we recover the Morse equalities $\chi_{+}[g] \equiv \chi(M) \; (\operatorname{mod} 2)$ (for $n$ even) and $\chi_{+}[g] = \chi(M)$ (for $n$ odd) due to Morse and van Schaack \cite[Theorem 10]{ms}.
\end{remark}

\begin{figure}[htbp]
\centering
\fbox{\begin{tikzpicture}
    \draw (0, 0) node {\includegraphics[height=0.4\textwidth]{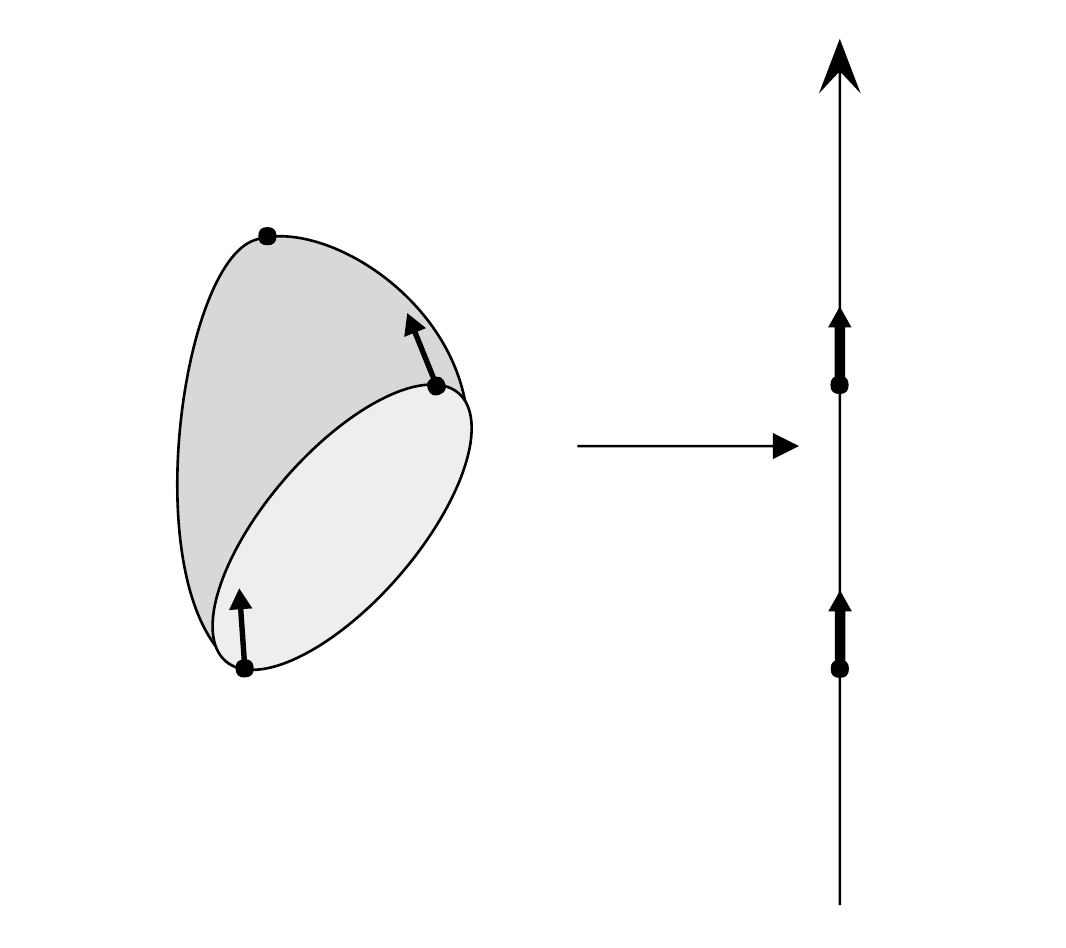}};
    \draw (2.1, 2.2) node {$\mathbb{R}$};
    \draw (2.5, 0.65) node {$df_{x_{1}}(v_{1})$};
    \draw (-1.5, -1.3) node {$x_{0}$};
    \draw (0.8, 0.4) node {$f$};
    \draw (-0.05, 0.4) node {$x_{1}$};
    \draw (2.5, -0.9) node {$df_{x_{0}}(v_{0})$};
    \draw (-2.2, 0.7) node {$D^{2}$};
\end{tikzpicture}}
\caption{Illustration of a Morse function $f \colon D^{2} \rightarrow \mathbb{R}$ induced by the height function on a $2$-disk embedded in $\mathbb{R}^{3}$.
The critical points of $f|_{S^{1}}$ are $x_{0}$ and $x_{1}$.
Using the indicated inward pointing tangent vectors $v_{0} \in T_{x_{0}}D^{2}$ and $v_{1} \in T_{x_{1}}D^{2}$, we see that $\sigma_{f}(x_{0}) = +1$ and $\sigma_{f}(x_{1}) = +1$.
Hence, $S_{+}^{0}[f] = \{x_{0}\}$, $S_{+}^{1}[f] = \{x_{1}\}$, and $\chi_{+}[f] = 0$.}
\label{figure morse generator example}
\end{figure}
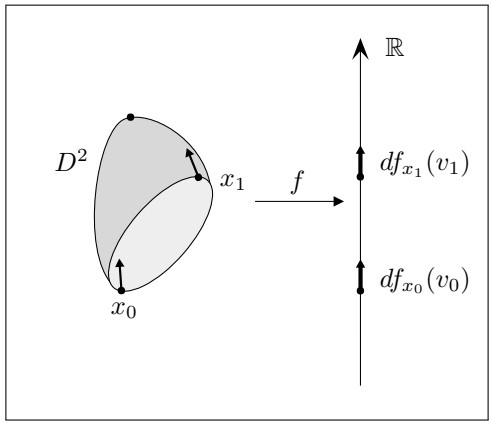

\section{Preliminaries on generic maps into the plane}\label{Preliminaries on generic maps into the plane}

The purpose of this section is to provide the necessary background on generic maps into the plane.
We explain Levine's cusp elimination technique (see \Cref{Elimination of cusps}) and the complementary process of creating pairs of cusps along fold lines (see \Cref{creation of cusps}).
These techniques are combined in \Cref{local modifications} to prove some useful lemmas for modifying generic maps locally.
In \Cref{generic extensions}, we discuss extendability of generic maps.

Let us first define generic maps into the plane by describing the occurring types of $C^{\infty}$ stable singular map germs explicitly (compare e.g. the introduction in \cite{lev}).

\begin{definition}[generic maps]\label{definition generic maps into the plane}
Let $G \colon X \rightarrow \mathbb{R}^{2}$ be a map defined on a manifold $X$ (without boundary) of dimension $n \geq 2$.
We call $G$ \emph{generic} if for every critical point $p \in S(G)$ of $G$ there exist coordinate charts centered at $p$ and $G(p)$, respectively, in which $G$ takes one of the following normal forms:
\begin{align}\label{definition fold and cusp}
(x_{1}, \dots, x_{n}) \mapsto \begin{cases}
(x_{1}, \pm x_{2}^{2} \pm \dots \pm x_{n}^{2}), \quad &\text{i.e., $p$ is a \emph{fold point} of $G$}, \\
(x_{1},x_{1}x_{2}+x_{2}^{3} \pm x_{3}^{2} \pm \dots \pm x_{n}^{2}), \quad &\text{i.e., $p$ is a \emph{cusp} of $G$}.
\end{cases}
\end{align}
Following \cite{lev} (see also Section 3 of \cite{ike}), we define the notion of an index for fold points and cusps of $G$ as follows.
If $p$ is a fold point of $G$, then the \emph{absolute index} of $p$ is defined as $\tau(p) = \operatorname{max}\{\lambda, n-1-\lambda\}$, where $\lambda$ denotes the number of minus signs that appear in the standard quadratic form $\pm x_{2}^{2} \pm \dots \pm x_{n}^{2}$ of fold points in (\ref{definition fold and cusp}).
If $p$ is a cusp of $G$, then the \emph{absolute index} of $p$ is defined as $\tau(p) = \operatorname{max}\{\lambda, n-2-\lambda\}$, where $\lambda$ denotes the number of minus signs that appear in the standard quadratic form $\pm x_{3}^{2} \pm \dots \pm x_{n}^{2}$ of cusps in (\ref{definition fold and cusp}).
In both cases, it turns out that $\tau(p)$ does not depend on the choice of coordinate charts around $p$ and $G(p)$.
\end{definition}

It is well-known that for a generic map $G \colon X \rightarrow \mathbb{R}^{2}$ as in the previous definition, the singular locus $S(G) \subset X$ is a submanifold of dimension $1$ which is closed as a subset, and the cusps of $G$ form a discrete subset of $X$.
Moreover, $G$ restricts to an immersion on the \emph{fold lines}, i.e., the components of the $1$-dimensional locus of fold points of $G$.
The absolute index is known to be constant along fold lines.

\begin{remark}[behavior of index]\label{remark index}
As stated in Lemma (3.2)(2) in \cite[p. 274]{lev}, the absolute index of fold points varies as follows in a neighborhood of a cusp $p$ of a generic map $G \colon X \rightarrow \mathbb{R}^{2}$.
Let $C$ denote the component of $S(G)$ which contains $p$.
If $n$ is even and $\tau(p) = \frac{n}{2} - 1$, then the two fold lines abutting $p$ on $C$ have absolute index $\frac{n}{2}$.
If $\tau(p) \neq \frac{n}{2} - 1$, then the two fold lines abutting $p$ on $C$ have absolute indices $\tau(p)$ and $\tau(p)+1$.
\end{remark}

\subsection{Elimination of cusps}\label{Elimination of cusps}
In this section, we review Levine's cusp elimination technique \cite{lev} (see also Section 3 in \cite{ike}) that forms a basis of our approach.

Consider a generic map $G \colon X \rightarrow \mathbb{R}^{2}$ on an $n$-manifold $X$ (without boundary).
Using the standard orientation on $\mathbb{R}^{2}$, Levine \cite{lev} assigns to every cusp $p$ of $G$ a \emph{(non-reduced) index} $I(p) \in \{0, \dots, n-2\}$, which is related to the absolute index by $\tau(p) = \operatorname{max}\{I(p), n-2-I(p)\}$, and can be characterized via normal forms as follows.

\begin{lemma}[index of cusps]\label{cusp orientation preserving}
Let $p$ be a cusp of the generic map $G \colon X^{n} \rightarrow \mathbb{R}^{2}$.
Then, there exist charts $\phi \colon U \rightarrow U' \subset \mathbb{R}^{n}$ centered at $p \in U$ and $\psi \colon V \rightarrow V' \subset \mathbb{R}^{2}$ centered at $G(p) \in V$ such that $G(U) \subset V$, $\psi$ is an \emph{orientation preserving} diffeomorphism, and the composition $\psi \circ G \circ \phi^{-1} \colon U' \rightarrow V'$ has the form
\begin{displaymath}
(\psi \circ G \circ \phi^{-1})(x_{1}, \dots, x_{n}) = (x_{1}, x_{1}x_{2} + x_{2}^{3} -\sum_{i = 1}^{k}x_{i+2}^{2} + \sum_{i=k+1}^{n-2}x_{i+2}^{2})
\end{displaymath}
for some $k \in \{0, \dots, n-2\}$.
In this situation, we have $I(p) = k$.
If, instead, $\psi$ is orientation reversing, then $I(p) = n-2-k$.
\end{lemma}

\begin{proof}
The desired charts exist by \Cref{definition generic maps into the plane}, where we possibly have to compose $\phi$ with $(x_{1}, x_{2}, \dots, x_{n}) \mapsto (x_{1}, -x_{2}, \dots, x_{n})$, and $\psi$ with $(a, b) \mapsto (a, -b)$.

The claims about the index are discussed after the definition of matching pairs in Section (4.3) in \cite[p. 284]{lev}.
\end{proof}

\begin{definition}\label{definition matching pair}
A pair of cusps $(p, p')$ of $G$ is called \emph{matching pair} if
\begin{displaymath}
I(p) + I(p') = n-2.
\end{displaymath}
A matching pair $(p, p')$ of $G$ is called a \emph{removable pair} (see Definition (4.5) in \cite[p. 285]{lev}) if there exists a joining curve for $(p, p')$, which is defined in Section (4.4) in \cite[p. 285]{lev} as an embedding $\lambda \colon [0, 1] \rightarrow X$ with $\lambda(0) = p$, $\lambda(1) = p'$, and $\lambda^{-1}(S(G)) = \{p, p'\}$, such that the composition $G \circ \lambda$ is an immersion that follows the direction of the cusps at the endpoints of $[0, 1]$ as indicated in \Cref{figure levine cusp elimination}.
\end{definition}

Note that, if $n \geq 3$ and $X$ is connected, then any matching pair of cusps of $G$ is automatically a removable pair by Lemma (1) in Section (4.4) in \cite[p. 285]{lev}.

The following is Levine's theorem on elimination of cusps (see \cite[p. 286ff]{lev}).

\begin{theorem}[Levine \cite{lev}]\label{proposition elimination of cusps}
Every matching pair $(p, p')$ of a generic map $G_{0} \colon X \rightarrow \mathbb{R}^{2}$ that is also removable can be eliminated (see \Cref{figure levine cusp elimination}).
More precisely, if $\lambda \colon [0, 1] \rightarrow X$ is a joining curve for $(p, p')$ as depicted in \Cref{figure levine cusp elimination}(a), then the process of cusp elimination modifies $G_{0}$ only in a prescribed neighborhood of $\lambda([0, 1]) \subset X$ (compare the lemma in Section (4.9) in \cite[p. 293]{lev}).
The image $G(S(G))$ of the singular set $S(G)$ of the modified generic map $G \colon X \rightarrow \mathbb{R}^{2}$ is indicated in \Cref{figure levine cusp elimination}(b).
\begin{figure}[htbp]
\centering
\fbox{\begin{tikzpicture}
    \draw (0, 0) node {\includegraphics[width=0.8\textwidth]{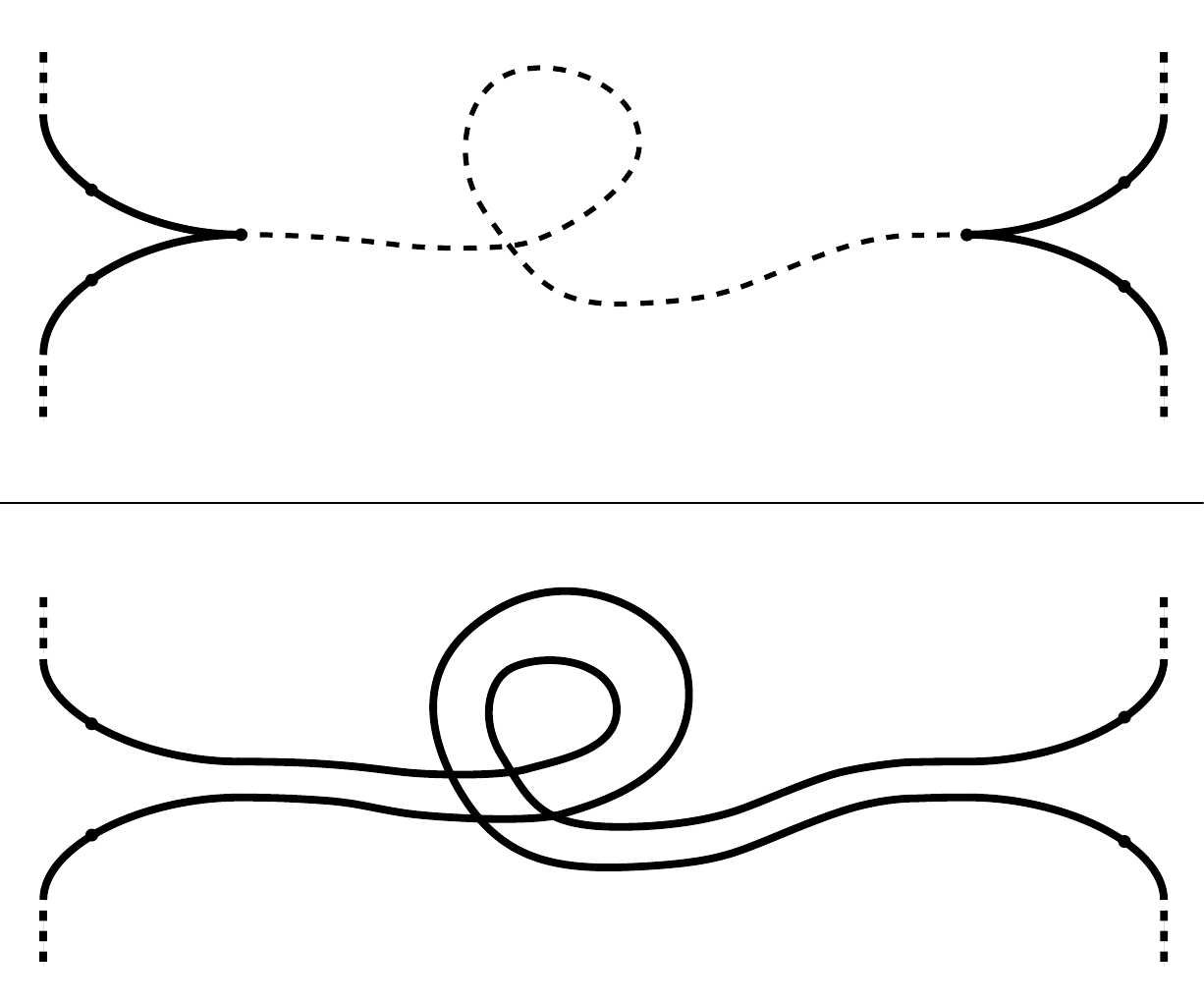}};
    \draw (-4.9, 4.2) node {$(a)$};
    \draw (-2.7, 1.9) node {$G_{0}(p)$};
    \draw (2.8, 1.9) node {$G_{0}(p')$};
    \draw (1.4, 3.0) node {$G_{0}(\lambda([0, 1]))$};
    \draw (0.2, 0.9) node {$\mathbb{R}^{2}$};
    \draw (-3.8, 2.9) node {$G_{0}(p_{1})$};
    \draw (-3.8, 1.6) node {$G_{0}(p_{2})$};
    \draw (3.8, 2.9) node {$G_{0}(p_{1}')$};
    \draw (3.8, 1.6) node {$G_{0}(p_{2}')$};

    \draw (-4.9, -0.4) node {$(b)$};
    \draw (0.2, -3.8) node {$\mathbb{R}^{2}$};
    \draw (-3.8, -1.7) node {$G(p_{1})$};
    \draw (-3.8, -3.0) node {$G(p_{2})$};
    \draw (3.8, -1.7) node {$G(p_{1}')$};
    \draw (3.8, -3.0) node {$G(p_{2}')$};
\end{tikzpicture}}
\caption{(a) A matching pair $(p, p')$ of $G_{0} \colon X^{n} \rightarrow \mathbb{R}^{2}$ connected by a joining curve $\lambda \colon [0, 1] \rightarrow X$.
(b) After elimination of $(p, p')$ by a homotopy supported in a small neighborhood of $\lambda([0, 1])$, the fold points $p_{1}, p_{2}$ near $p$, and $p_{1}', p_{2}'$ near $p'$ are connected by the fold lines of the modified generic map $G$ as indicated.}
\label{figure levine cusp elimination}
\end{figure}
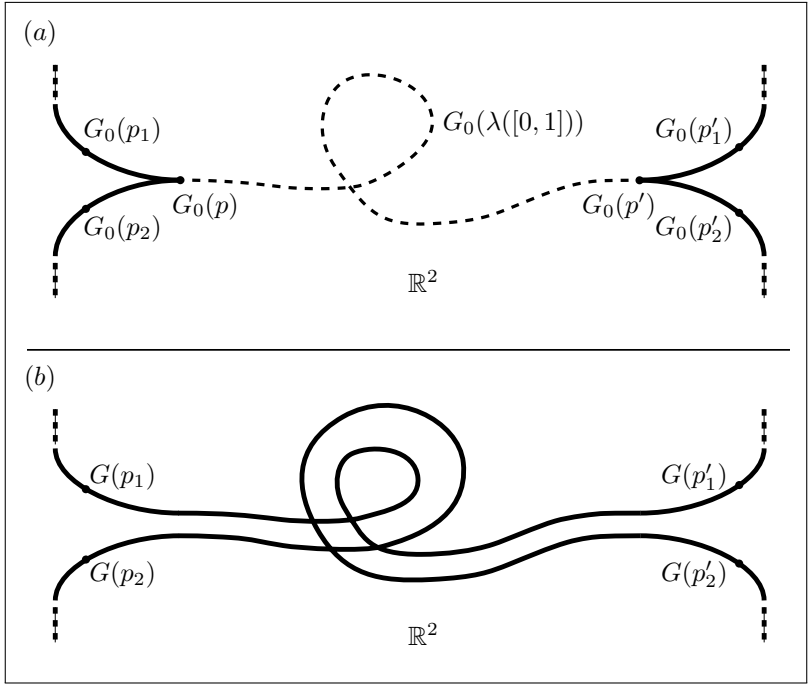
\end{theorem}

\subsection{Creation of cusps}\label{creation of cusps}
Given a generic map $G \colon X^{n} \rightarrow \mathbb{R}^{2}$ on an $n$-manifold $X$ (without boundary) of dimension $n \geq 2$, we discuss the process of creating a pair of cusps on a given fold line of $G$.
The local model for this modification is provided by the swallow's tail homotopy (compare Exercise (3) of Section VII.\S 3 in \cite[p. 176]{gol}).

\begin{proposition}\label{proposition creation of cusps}
Given $i \in \{0, \dots, n-2\}$ and an open neighborhood $N \subset \mathbb{R}^{n}$ of the origin $0 \in \mathbb{R}^{n}$, there exist a generic map $F \colon \mathbb{R}^{n} \rightarrow \mathbb{R}^{2}$, a compact subset $L \subset N$, and an embedding $\xi \colon \mathbb{R} \rightarrow \mathbb{R}^{n}$ with the following properties:
\begin{enumerate}[(i)]
\item The generic map $F$ agrees on $\mathbb{R}^{n} \setminus L$ with the normal form of fold points,
$$
F(x_{1}, \dots, x_{n}) = (x_{1}, -x_{2}^{2} - \dots - x_{i+1}^{2} + x_{i+2}^{2} + \dots + x_{n}^{2}).
$$
\item The singular set of $F$ is given by the image of $\xi$, $S(F) = \xi(\mathbb{R})$.
\item If $|s| < 1$, then $\xi(s)$ is a fold point of $F$ of absolute index $\operatorname{max}\{i+1, n-2-i\}$.
\item The par $(\xi(-1), \xi(1))$ is a matching pair of cusps of $F$ (see \Cref{definition matching pair}).
\item If $|s| > 1$, then $\xi(s)$ is a fold point of $F$ of absolute index $\operatorname{max}\{i, n-1-i\}$.
\end{enumerate}
\end{proposition}

\begin{proof}
We consider the swallow's tail homotopy
\begin{align*}
H_{t} \colon \mathbb{R}^{n} = \mathbb{R} \times \mathbb{R} \times \mathbb{R}^{n-2} \rightarrow \mathbb{R}^{2}, \qquad t \in \mathbb{R}, \\
p = (u, x, z) \mapsto (u, \frac{x^{4}}{12} - t \frac{x^{2}}{2} + ux + Q(z)) = (u, h_{t}(p)),
\end{align*}
where $Q(z) = - z_{1}^{2} - \dots -z_{i}^{2} + z_{i+1}^{2} + \dots + z_{n-2}^{2}$ denotes the standard quadratic form of index $i$ in $n-2$ variables $z = (z_{1}, \dots, z_{n-2})$.
According to Lemma 4.7.1 in \cite[p. 110]{wra}, the singular set of $H_{t}$, $t \in \mathbb{R}$, is given by the image of the embedding
\begin{displaymath}
\varphi_{t} \colon \mathbb{R} \rightarrow \mathbb{R}^{n} = \mathbb{R} \times \mathbb{R} \times \mathbb{R}^{n-2}, \qquad \varphi_{t}(x) = (-\frac{x^{3}}{3} + tx, x, 0).
\end{displaymath}
Furthermore, if $t < 0$, then the singular set $S(H_{t}) = \varphi_{t}(\mathbb{R})$ is a fold line of $H_{t}$ which has absolute index $\operatorname{max}\{i, n-1-i\}$.
If $t > 0$, then $(\varphi_{t}(-\sqrt{t}), \varphi_{t}(\sqrt{t}))$ is a matching pair of cusps of $H_{t}$ (see \Cref{definition matching pair}), the points $\varphi_{t}(x)$ for $|x| > \sqrt{t}$ are fold points of $H_{t}$ of absolute index $\operatorname{max}\{i, n-1-i\}$, and the points $\varphi_{t}(x)$ for $|x| < \sqrt{t}$ are fold points of $H_{t}$ of absolute index $\operatorname{max}\{i+1, n-2-i\}$.

As the origin $\varphi_{-1}(0) = 0 \in \mathbb{R}^{n}$ is a fold point of absolute index $\operatorname{max}\{i, n-1-i\}$ of the fold map $H_{-1}$, there exist a chart $\alpha \colon U \rightarrow U' \subset \mathbb{R}^{n}$ centered at $\varphi_{-1}(0) = 0 \in U$ and a chart $\beta \colon V \rightarrow V' \subset \mathbb{R}^{2}$ centered at $H_{-1}(0) = 0 \in V$ such that $H_{-1}(U) \subset V$, and for all $(x_{1}, \dots, x_{n}) \in U'$ we have
$$
(\beta \circ H_{-1} \circ \alpha^{-1})(x_{1}, \dots, x_{n}) = (x_{1}, -x_{2}^{2} - \dots - x_{i+1}^{2} + x_{i+2}^{2} + \dots + x_{n}^{2}).
$$
Applying Proposition 4.7.3 in \cite[p. 111]{wra} to the open neighborhood $\alpha^{-1}(N \cap U')$ of the origin $0 \in \mathbb{R}^{n}$, we obtain a generic map $G \colon \mathbb{R}^{n} \rightarrow \mathbb{R}^{2}$, a compact subset $K \subset \alpha^{-1}(N \cap U')$ and an embedding $\varphi \colon \mathbb{R} \rightarrow \mathbb{R}^{n}$ with the following properties:
\begin{enumerate}[(i)]
\item For all $p \in \mathbb{R}^{n} \setminus K$, we have $G(p) = H_{-1}(p)$.
\item The singular set of $G$ is given by the image of $\varphi$, $S(G) = \varphi(\mathbb{R})$.
\item If $|s| < 1$, then $\varphi(s)$ is a fold point of $G$ of absolute index $\operatorname{max}\{i+1, n-2-i\}$.
\item The pair $(\varphi(-1), \varphi(1))$ is a matching pair of cusps of $G$.
\item If $|s| > 1$, then $\varphi(s)$ is a fold point of $G$ of absolute index $\operatorname{max}\{i, n-1-i\}$.
\end{enumerate}
Note that $L = \alpha(K)$ is a compact subset of $N$.
We define the generic map $F \colon \mathbb{R}^{n} \rightarrow \mathbb{R}^{2}$ by $F(x_{1}, \dots, x_{n}) = (\beta \circ G \circ \alpha^{-1})(x_{1}, \dots, x_{n})$ for all $(x_{1}, \dots, x_{n}) \in U'$ and
$$
F(x_{1}, \dots, x_{n}) = (x_{1}, -x_{2}^{2} - \dots - x_{i+1}^{2} + x_{i+2}^{2} + \dots + x_{n}^{2})
$$
for all $(x_{1}, \dots, x_{n}) \in \mathbb{R}^{n} \setminus L$.
Then, by construction, $S(F) \setminus L = \{(x_{1}, 0, \dots, 0) \in \mathbb{R}^{n}\} \setminus L$ and $S(F) \cap U' = \alpha(S(G) \cap U)$, and thus $S(F) \cong \mathbb{R}$.
Since $\alpha$ and $\beta$ are diffeomorphisms, the claimed properties (ii) to (v) are valid for a suitable diffeomorphism $\xi \colon \mathbb{R} \rightarrow S(F)$.
\end{proof}

\subsection{Local modifications}\label{local modifications}

We apply the techniques of elimination and creation of cusps from \Cref{Elimination of cusps} and \Cref{creation of cusps} to derive some specific local modifications of generic maps that will be used in the proofs of \Cref{proposition condition n even} and \Cref{proposition condition n odd}.

\begin{lemma}\label{lemma cusp number change in even dimensions}
Let $n \geq 2$ be an even integer.
Let $G_{0} \colon X^{n} \rightarrow \mathbb{R}^{2}$ be a generic map on an $n$-manifold $X$ (without boundary).
Suppose that $C_{0}$ is a component of $S(G_{0})$ that contains a finite number of cusps of $G_{0}$.
Let $U \subset X$ be an open subset such that $U \cap C_{0} \neq \emptyset$.
Then, there exist a generic map $G \colon X \rightarrow \mathbb{R}^{2}$ and a component $C$ of $S(G)$ with the following properties.
We have $G|_{X \setminus K} = G_{0}|_{X \setminus K}$ and $C \setminus K = C_{0} \setminus K$ for some compact subset $K \subset U$, and $C$ contains a finite number of cusps of $G$ that has not the same parity as the number of cusps of $G_{0}$ lying on $C_{0}$.
\begin{figure}[htbp]
\centering
\fbox{\begin{tikzpicture}
    \draw (0, 0) node {\includegraphics[height=0.58\textwidth]{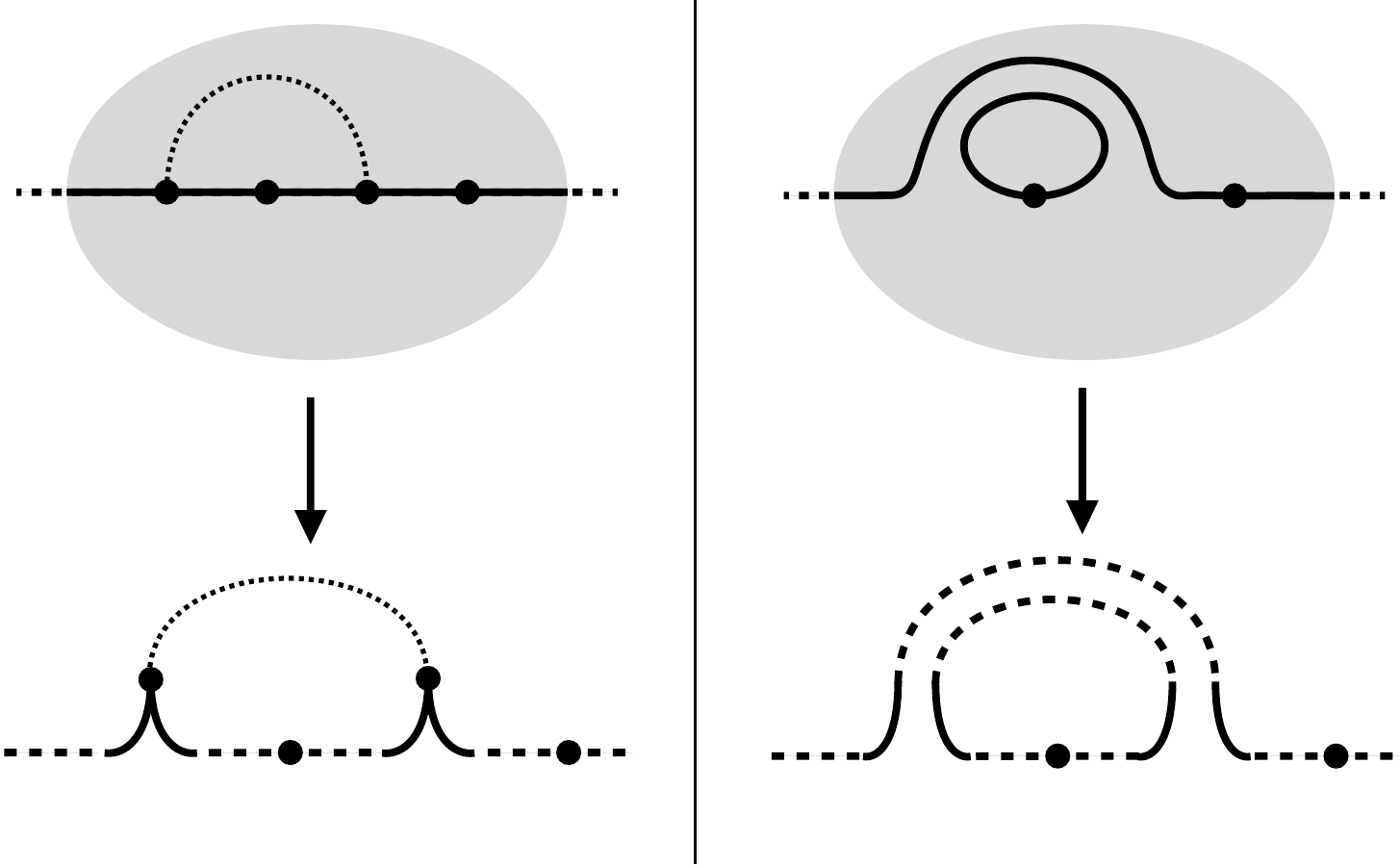}};
    \draw (-5.7, 3.4) node {$(a)$};
    \draw (-0.7, 2.4) node {$C_{2}$};
    \draw (-2.8, 3.0) node {$\lambda$};
    \draw (-4.5, 1.7) node {$c_{0}$};
    \draw (-2.8, 1.7) node {$c_{0}'$};
    \draw (-3.65, 1.7) node {$c_{1}$};
    \draw (-1.95, 1.7) node {$c_{1}'$};
    \draw (-1.3, 1.0) node {$U$};
    \draw (-3.7, -0.2) node {$G_{2}$};
    \draw (-4.7, -1.2) node {$G_{2} \circ \lambda$};
    \draw (-5.35, -2.1) node {$G_{2}(c_{0})$};
    \draw (-1.55, -2.1) node {$G_{2}(c_{0}')$};
    \draw (-5.3, -3.1) node {$G_{2} \circ \alpha$};
    \draw (-3.3, -3.0) node {$G_{2}(c_{1})$};
    \draw (-1.2, -3.0) node {$G_{2}(c_{1}')$};
    \draw (-1.9, -1.2) node {$\mathbb{R}^{2}$};

    \draw (0.45, 3.4) node {$(b)$};
    \draw (5.75, 2.4) node {$C$};
    \draw (5.15, 1.0) node {$U$};
    \draw (2.8, 1.7) node {$c$};
    \draw (4.5, 1.7) node {$c'$};
    \draw (2.9, -0.2) node {$G$};
    \draw (3.15, -3.05) node {$G(c)$};
    \draw (5.25, -3.05) node {$G(c')$};
    \draw (4.7, -1.2) node {$\mathbb{R}^{2}$};
\end{tikzpicture}}
\caption{Local modifications of the generic map $G_{0}$ on $U \subset X$.
(a) Along a fold line of $G_{1}|_{U}$ of absolute index $n/2$, we introduce the cusp pairs $(c_{0}, c_{1})$ and $(c_{0}', c_{1}')$ having absolute indices $(n-2)/2$.
(b) We eliminate the matching pair $(c_{0}, c_{0}')$ of cusps of $G_{2}$ by means of a joining curve $\lambda$ to obtain the desired generic map $G$.}
\label{figure cusp elimination}
\end{figure}
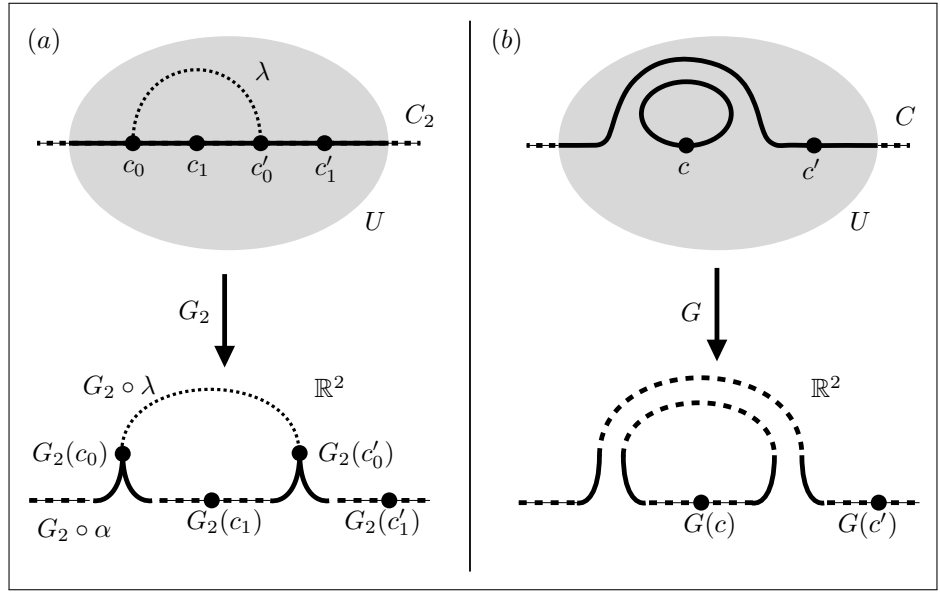
\end{lemma}

\begin{proof}
Applying \Cref{proposition creation of cusps} iteratively, we modify $G_{0}|_{U}$ by creating several new pairs of cusps along $U \cap C_{0}$ in order to obtain a modified generic map $G_{1} \colon X^{n} \rightarrow \mathbb{R}^{2}$ and a component $C_{1}$ of $S(G_{1})$ with the following properties.
We have $G_{1}|_{X \setminus K_{1}} = G_{0}|_{X \setminus K_{1}}$ and $C_{1} \setminus K_{1} = C_{0} \setminus K_{1}$ for some compact subset $K_{1} \subset U$, the intersection $U \cap C_{1}$ contains fold points of $G_{1}$ of absolute index $n/2$ (where we note that $n \geq 2$ is even), and $C_{1}$ contains a finite number of cusps of $G_{1}$ that has the same parity as the number of cusps of $G_{0}$ lying on $C_{0}$.

Next, we apply \Cref{proposition creation of cusps} twice to create two pairs $(c_{0}, c_{1})$ and $(c_{0}', c_{1}')$ of cusps having absolute indices $(n-2)/2$ along a fold line of $G_{1}|_{U}$ of absolute index $n/2$ in such a way that we obtain a modified generic map $G_{2} \colon X^{n} \rightarrow \mathbb{R}^{2}$ and a component $C_{2}$ of $S(G_{2})$ with the following properties (see \Cref{figure cusp elimination}(a)).
We have $G_{2}|_{X \setminus K_{2}} = G_{1}|_{X \setminus K_{2}}$ and $C_{2} \setminus K_{2} = C_{1} \setminus K_{2}$ for some compact subset $K_{2} \subset U$, and there exists an embedding $\alpha \colon (0, 1) \rightarrow U \cap C_{2}$ such that
\begin{itemize}
\item the number of cusps of $G_{2}$ lying on $C_{2} \setminus \alpha((0, 1))$ equals the number of cusps of $G_{1}$ lying on $C_{1}$,
\item there are $4$ cusps of $G_{2}$ that lie on $\alpha((0, 1))$ with each having absolute index $(n-2)/2$, namely $c_{0} = \alpha(t_{0})$, $c_{1} = \alpha(t_{1})$, $c_{0}' = \alpha(t_{0}')$, $c_{1}' = \alpha(t_{1}')$ for some real numbers $0 < t_{0} < t_{0}' < t_{1} < t_{1}' < 1$, and
\item the composition $G_{2} \circ \alpha \colon (0, 1) \rightarrow \mathbb{R}^{2}$ looks as depicted in \Cref{figure cusp elimination}(a).
\end{itemize}

Finally, we apply \Cref{proposition elimination of cusps} to $G_{2}|_{U}$ by eliminating the matching pair $(c_{0}, c_{0}')$.
(Note that $(c_{0}, c_{0}')$ is in fact removable because a joining curve $\lambda$ for $(c_{0}, c_{0}')$ exists in a tubular neighborhood of $\alpha((0, 1)) \subset U$ whenever $n>2$, and also in the case $n=2$ as can be seen from \Cref{figure cusp elimination}(a).)
As can be seen from \Cref{figure cusp elimination}(b), we obtain a modified generic map $G \colon X \rightarrow \mathbb{R}^{2}$ and a component $C$ of $S(G)$ with the following properties.
We have $G|_{X \setminus K_{3}} = G_{2}|_{X \setminus K_{3}}$ and $C \setminus K_{3} = C_{2} \setminus K_{2}$ for some compact subset $K_{3} \subset U$, and $C$ contains a finite number of cusps of $G$ that has not the same parity as the number of cusps of $G_{2}$ lying on $C_{2}$.

All in all, it follows that $G$ and $C$ have the desired properties.
\end{proof}

\begin{lemma}\label{lemma surgery of singular components}
Let $n > 2$ be an odd integer.
Let $G_{0} \colon X^{n} \rightarrow \mathbb{R}^{2}$ be a generic map on an $n$-manifold $X$ (without boundary).
Suppose that $C_{0}^{(0)}$ and $C_{0}^{(1)}$ are two distinct components of $S(G_{0})$, and
let $U \subset X$ be a connected open subset such that $U \cap S(G_{0}) \subset C_{0}^{(0)} \cup C_{0}^{(1)}$, and $U \cap C_{0}^{(0)} \neq \emptyset$, $U \cap C_{0}^{(1)} \neq \emptyset$.
Given points $x^{(0)} \in C_{0}^{(0)} \setminus U$ and $x^{(1)} \in C_{0}^{(1)} \setminus U$, there exist a generic map $G \colon X \rightarrow \mathbb{R}^{2}$ such that $G|_{X \setminus K} = G_{0}|_{X \setminus K}$ for some compact subset $K \subset X$, and the points $x^{(0)}$ and $x^{(1)}$ lie on a common component $C$ of $S(G)$.
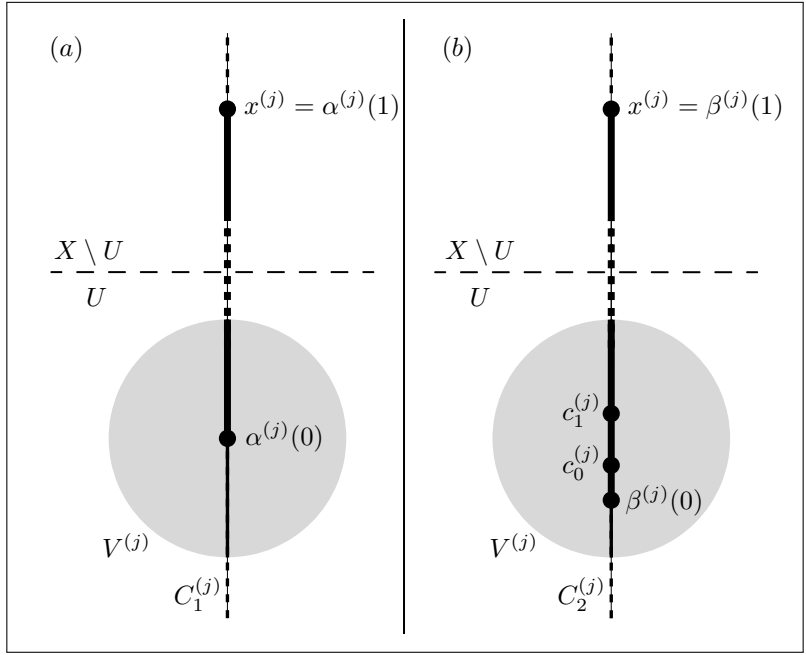
\begin{figure}[htbp]
\centering
\fbox{\begin{tikzpicture}
    \draw (0, 0) node {\includegraphics[width=0.8\textwidth]{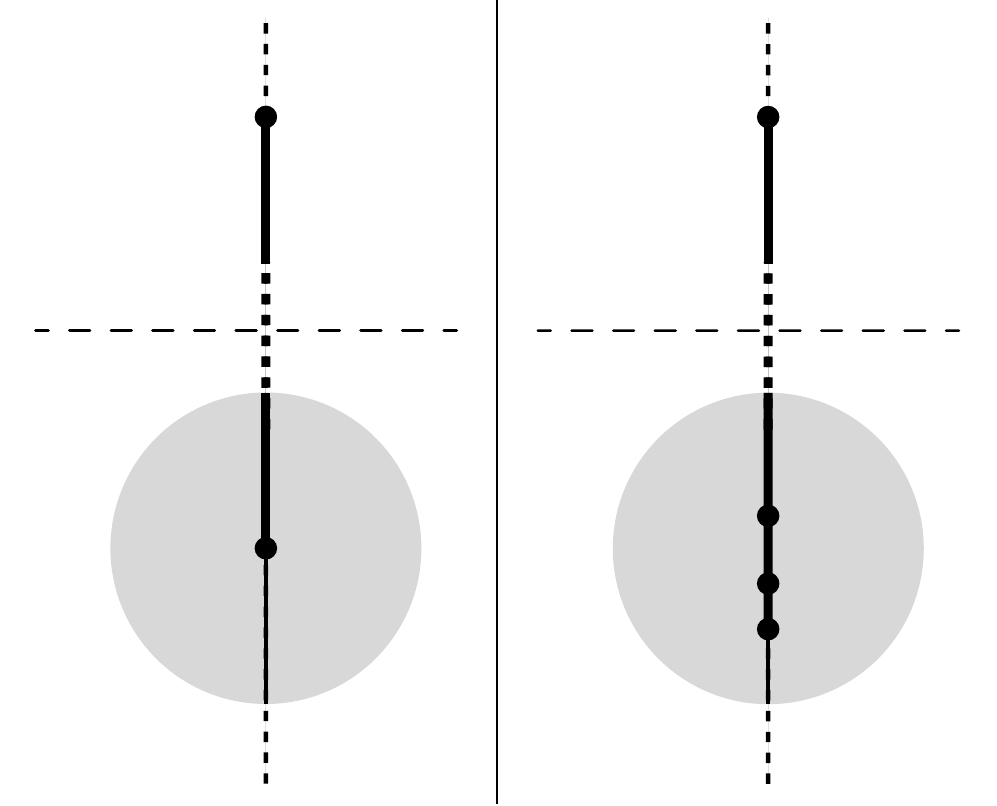}};
    \draw (-4.5, 3.7) node {$(a)$};
    \draw (0.7, 3.7) node {$(b)$};

    \draw (-2.75, -3.55) node {$C_{1}^{(j)}$};
    \draw (2.35, -3.55) node {$C_{2}^{(j)}$};

    \draw (-3.7, -2.9) node {$V^{(j)}$};
    \draw (1.45, -2.9) node {$V^{(j)}$};

    \draw (-1.6, -1.45) node {$\alpha^{(j)}(0)$};
    \draw (-1.1, 2.95) node {$x^{(j)} = \alpha^{(j)}(1)$};

    \draw (3.45, -2.3) node {$\beta^{(j)}(0)$};
    \draw (4.0, 2.95) node {$x^{(j)} = \beta^{(j)}(1)$};

    \draw (-4.1, 0.4) node {$U$};
    \draw (-4.2, 1.0) node {$X \setminus U$};

    \draw (1.0, 0.4) node {$U$};
    \draw (1.0, 1.0) node {$X \setminus U$};

    \draw (2.35, -1.8) node {$c^{(j)}_{0}$};
    \draw (2.35, -1.1) node {$c^{(j)}_{1}$};
\end{tikzpicture}}
\caption{(a) The embedding $\alpha^{(j)} \colon [0, 1] \rightarrow C_{1}^{(j)}$ is chosen to connect a fold point $\alpha^{(j)}(0) \in U$ of $G_{1}$ of absolute index $(n-1)/2$ with the point $\alpha^{(j)}(1) = x^{(j)} \in X \setminus U$.
(b) We create the cusps $(c_{0}^{(j)}, c_{1}^{(j)})$ near the fold point $\alpha^{(j)}(0)$ of $G_{1}|_{V^{(j)}}$ in such a way that $(c_{1}^{(0)}, c_{1}^{(1)})$ is a matching pair, and that the cusps $(c_{0}^{(j)}, c_{1}^{(j)})$ lie in the indicated order on the image of an embedding $\beta^{(j)} \colon [0, 1] \rightarrow C_{2}^{(j)}$ that connects a fold point $\beta^{(j)}(0) \in V^{(j)}$ of $G_{2}$ of absolute index $(n-1)/2$ with the point $\beta^{(j)}(1) = x^{(j)} \in X \setminus U$.}
\label{figure local component}
\end{figure}
\end{lemma}

\begin{proof}
Applying \Cref{proposition creation of cusps} iteratively, we modify $G_{0}|_{U}$ by creating several new pairs of cusps along $U \cap C_{0}^{(0)}$ and $U \cap C_{0}^{(1)}$ in order to obtain a modified generic map $G_{1} \colon X^{n} \rightarrow \mathbb{R}^{2}$ and two distinct components $C_{1}^{(0)}$ and $C_{1}^{(1)}$ of $S(G_{1})$ with the following properties.
We have $G_{1}|_{X \setminus K_{1}} = G_{0}|_{X \setminus K_{1}}$ and $C_{1}^{(j)} \setminus K_{1} = C_{0}^{(j)} \setminus K_{1}$, $j = 0, 1$, for some compact subset $K_{1} \subset U$, and the intersections $U \cap C_{1}^{(j)}$, $j=0, 1$, contain fold points of $G_{1}$ of absolute index $(n-1)/2$ (where we note that $n > 2$ is odd).

As shown in \Cref{figure local component}(a), we choose embeddings $\alpha^{(j)} \colon [0, 1] \rightarrow C_{1}^{(j)}$, $j=0, 1$, such that $\alpha^{(j)}(0) \in U \cap C_{1}^{(j)}$ is a fold point of $G_{1}$ of absolute index $i = (n-1)/2$, and $\alpha^{(j)}(1) = x^{(j)}$.
For $j = 0, 1$, there exist charts $\varphi^{(j)} \colon V^{(j)} \rightarrow \overline{V}^{(j)} \subset \mathbb{R}^{n}$ centered at $\alpha^{(j)}(0) \in U$ and $\psi^{(j)} \colon W^{(j)} \rightarrow \overline{W}^{(j)} \subset \mathbb{R}^{2}$ centered at $G_{1}(\alpha^{(j)}(0)) \in \mathbb{R}^{2}$ such that $G_{1}(V^{(j)}) \subset W^{(j)}$, and the composition $\psi^{(j)} \circ G_{1} \circ (\varphi^{(j)})^{-1}$ satisfies
$$
(\psi^{(j)} \circ G_{1} \circ (\varphi^{(j)})^{-1})(x_{1}, \dots, x_{n}) = (x_{1}, -x_{2}^{2} - \dots - x_{i + 1}^{2} + x_{i + 2}^{2} + \dots + x_{n}^{2})
$$
for all $(x_{1}, \dots, x_{n}) \in \overline{V}^{(j)}$.
Moreover, we may assume without loss of generality that $\varphi^{(j)}(V^{(j)} \cap \alpha^{(j)}([0, 1])) = \{(x_{1}, 0, \dots, 0) \in \mathbb{R}^{n}; \; (-1)^{j} \cdot x_{1} \geq 0\} \cap \overline{V}$, and that $\psi^{(j)}$ is orientation preserving.
(If necessary, we compose $\psi^{(j)}$ with one of the automorphisms $(a, b) \mapsto (a, -b)$ and $(a, b) \mapsto (-a, b)$ and $(a, b) \mapsto (-a, -b)$ of $\mathbb{R}^{2}$.)

Next, we apply \Cref{proposition creation of cusps} for $i = (n-1)/2$ to the open subsets $\overline{V}^{(j)} \subset \mathbb{R}^{n}$, $j=0, 1$, to create a pair $(c_{0}^{(j)}, c_{1}^{(j)})$ of cusps near the fold point $\alpha^{(j)}(0)$ of $G_{1}|_{V^{(j)}}$ (where we may assume that $V^{(0)} \cap V^{(1)} = \emptyset$) in such a way that we obtain a modified generic map $G_{2} \colon X^{n} \rightarrow \mathbb{R}^{2}$ and two distinct components $C_{2}^{(0)}$ and $C_{2}^{(1)}$ of $S(G_{2})$ with the following properties (see \Cref{figure local component}(b)).
We have $G_{2}|_{X \setminus (V^{(0)} \cup V^{(1)})} = G_{1}|_{X \setminus (V^{(0)} \cup V^{(1)})}$ and $C_{2}^{(j)} \setminus V^{(j)} = C_{1}^{(j)} \setminus V^{(j)}$, $j = 0, 1$, and there exist embeddings $\beta^{(j)} \colon [0, 1] \rightarrow C_{2}^{(j)}$, $j=0, 1$, such that
\begin{itemize}
\item we have $\beta^{(j)}(0) \in V^{(j)}$ and $\beta^{(j)}(1) = x^{(j)}$,
\item there are $2$ cusps of $G_{2}$ that lie on $V^{(j)}$, namely $c_{0}^{(j)} = \beta^{(j)}(t_{0}^{(j)})$ and $c_{1}^{(j)} = \beta^{(j)}(t_{1}^{(j)})$ for some real numbers $0 < t_{0}^{(j)} < t_{1}^{(j)} < 1$, and for the absolute index we have $\tau(c_{0}^{(j)}) = \tau(c_{1}^{(j)}) = i = (n-1)/2$, and
\item the pair $(c_{1}^{(0)}, c_{1}^{(1)})$ is a matching pair in the sense of \Cref{definition matching pair}.
\end{itemize}

Finally, we apply \Cref{proposition elimination of cusps} to $G_{2}|_{U}$ by eliminating the matching pair $(c_{1}^{(0)}, c_{1}^{(1)})$.
(Note that $(c_{1}^{(0)}, c_{1}^{(1)})$ is in fact removable because $n > 2$, and $U$ is connected by assumption.)
Thus, we obtain a modified generic map $G \colon X \rightarrow \mathbb{R}^{2}$ and a component $C$ of $S(G)$ with the following properties.
We have $G|_{X \setminus K} = G_{2}|_{X \setminus K}$ for some compact subset $K \subset U$, the points $x^{(0)}$ and $x^{(1)}$ lie both on $C$.

All in all, it follows that $G$ and $C$ have the desired properties.
\end{proof}

\subsection{Generic extensions}\label{generic extensions}
It is a well-known fact that any map $X \rightarrow \mathbb{R}^{2}$ can be approximated arbitrarily well by a generic map as in \Cref{definition generic maps into the plane} (see \cite{lev}).
We will need the following extension property for generic maps.

\begin{proposition}\label{proposition generic extension}
Let $X$ be a manifold (without boundary).
Suppose that $G_{0} \colon X \setminus C \rightarrow \mathbb{R}^{2}$ is a generic map, where $C \subset X$ is compact.
Then, given an open neighborhood $U \subset X$ of $C$, there exists a generic map $G \colon X \rightarrow \mathbb{R}^{2}$ that extends $G_{0}|_{X \setminus U}$.
\end{proposition}

\begin{proof}
We note that fold points and cusps can be characterized by means of transversality and jet spaces (see e.g. \cite[Chapter III, \S 4]{gol}).
Thus, the proof is a standard application of a relative version of the Thom transversality theorem.
The required transversality techniques are for instance provided in \cite[Chapter II, \S 4 and \S 5]{gol}.

For an explicit elaboration of the details, see Proposition 4.4.1 in \cite[p. 100]{wra}.
\end{proof}

\begin{corollary}\label{corollary generic maps}
Given a compact manifold possibly with boundary $Y$ and a Morse function $g \colon \partial Y \rightarrow \mathbb{R}$, there exists a map $G \colon Y \rightarrow \mathbb{R}^{2}$ such that $G|_{Y \setminus \partial Y}$ is generic, and $G|_{[0, \varepsilon) \times \partial Y} = \operatorname{id}_{[0, \varepsilon)} \times g$ in some collar neighborhood $[0, \varepsilon) \times \partial Y \subset Y$ of $\partial Y \subset Y$.
\end{corollary}

\begin{proof}
Choose a collar neighborhood $[0, \infty) \times \partial Y \subset Y$ of $\partial Y \subset Y$.
By means of \Cref{proposition generic extension} we may then extend the $\{0\} \times \partial Y$-germ of the map $\operatorname{id}_{[0, \infty)} \times g \colon [0, \infty) \times \partial Y \rightarrow [0, \infty) \times \mathbb{R}$ to a map $G \colon Y \rightarrow \mathbb{R}^{2}$ such that $G|_{Y \setminus \partial Y}$ is generic.
\end{proof}

\section{Admissible extensions}\label{Non-singular extensions}
Let us adopt from \cite{sy3} the terminology that a map $M \rightarrow N$ between manifolds possibly with boundary is called \emph{admissible} if it is a submersion on a neighborhood of the boundary $\partial M$ of the source manifold $M$.
In particular, note that Morse functions $M \rightarrow \mathbb{R}$ are by definition admissible.

In this section, we study the extendability of generic maps $\partial M \rightarrow \mathbb{R}$ (see \Cref{remark realization of sign distribution}) and $\partial M \rightarrow \mathbb{R}^{2}$ (compare \Cref{proposition extension of G} and \Cref{proposition cusps and vector fields}) to admissible maps $M \rightarrow \mathbb{R}$ and $M \rightarrow \mathbb{R}^{2}$, respectively.

\begin{lemma}\label{remark realization of sign distribution}
Let $M^{n}$ be a compact $n$-manifold possibly with boundary.
Given a Morse function $g \colon \partial M \rightarrow \mathbb{R}$ and a map $\sigma \colon S(g) \rightarrow \{\pm 1\}$, there exists a Morse function $f \colon M \rightarrow \mathbb{R}$ such that $f|_{\partial M} = g$ and $\sigma_{f} = \sigma$ (see (\ref{sign function})).
\end{lemma}

\begin{proof}
The normal form of a non-degenerate critical point of index $i$,
$$
\phi \colon \mathbb{R}^{n-1} \rightarrow \mathbb{R}, \qquad x = (x_{1}, \dots, x_{n-1}) \mapsto c - x_{1}^{2} - \dots - x_{i}^{2} + x_{i+1}^{2} + \dots + x_{n-1}^{2},
$$
extends to a submersion
$$
\widetilde{\phi} \colon \mathbb{R}^{n-1} \times [0, \infty) \rightarrow \mathbb{R}, \qquad (x, t) \mapsto \phi(x) \pm t.
$$
Note that the sign of the summand $\pm t$ determines whether the image under $d \widetilde{\phi}_{(0, 0)}$ of an inward pointing tangent vector based at $(0, 0) \in \mathbb{R}^{n-1} \times [0, \infty)$ points into the positive or into the negative direction of the real axis.
Given a compact neighborhood $K \subset \mathbb{R}^{n-1}$ of the origin, there exists a submersion of the form
$$
\hat{\phi} \colon \mathbb{R}^{n-1} \times [0, \varepsilon) \rightarrow \mathbb{R}
$$
such that $\hat{\phi}(x, t) = \widetilde{\phi}(x, t)$ for all $x \in \mathbb{R}^{n-1}$ near the origin, and $\hat{\phi}(x, t) = \phi(x)$ for all $x \notin K$.
(In fact, for sufficiently small $\varepsilon > 0$, we can take $\hat{\phi}(x, t) = \phi(x) \pm t \cdot \alpha(x)$ for any function $\alpha \colon \mathbb{R}^{n-1} \rightarrow \mathbb{R}$ such that $\alpha(x) = 1$ for all $x \in \mathbb{R}^{n-1}$ near the origin and $\alpha(x) = 0$ for all $x \notin K$.)

Using submersions $\hat{\phi}$ of the above form in suitable charts around the critical points of $g$, we can extend $g \colon \partial M \rightarrow \mathbb{R}$ to a submersion $\hat{g} \colon \partial M \times [0, \varepsilon) \rightarrow \mathbb{R}$ defined on a collar neighborhood of $\partial M$ in $M$.
In addition, by choosing the appropriate signs for the summands $\pm t$, we can achieve that for every critical point $x \in S(g)$, the tangent vector $\sigma(x) \cdot d \hat{g}_{x}(v) \in T_{\hat{g}(x)} \mathbb{R} = \mathbb{R}$ points into the positive direction of the real axis for an inward pointing tangent vector $v \in T_{x}M$.
Finally, there is no obstruction to extending the $\partial M$-germ of the submersion $\hat{g}$ to the desired Morse function $f \colon M \rightarrow \mathbb{R}$.
\end{proof}

\begin{proposition}[see \Cref{figure proposition non-singular extension}]\label{proposition extension of G}
Let $V^{n}$ be a compact manifold possibly with boundary of dimension $n \geq 2$.
Suppose that $h \colon \partial V \times [0, \infty) \rightarrow \mathbb{R}$ is an admissible map that restricts to a Morse function $g = h|_{\partial V} \colon \partial V \rightarrow \mathbb{R}$ on $\partial V = \partial V \times \{0\}$.
In particular, the assignment $\sigma_{h} \colon S(g) \rightarrow \{\pm 1\}$ is defined as in (\ref{sign function}).
Let $G \colon V \rightarrow \mathbb{R}^{2}$ be a map such that $G|_{V \setminus \partial V}$ is generic, and $G|_{[0, \varepsilon) \times \partial V} = \operatorname{id}_{[0, \varepsilon)} \times g$ in some collar neighborhood $[0, \varepsilon) \times \partial V \subset V$ of $\partial V \subset V$.
Then, the following statements are equivalent:
\begin{enumerate}[(i)]
\item There exists a map $H \colon V \times [0, \infty) \rightarrow \mathbb{R}^{2}$ such that
\begin{itemize}
\item $H|_{V \times \{0\}} = G$,
\item $H|_{[0, \varepsilon') \times \partial V \times [0, \infty)} = \operatorname{id}_{[0, \varepsilon')} \times h$ for some $\varepsilon' \in (0, \varepsilon)$, and
\item $H|_{(V \setminus \partial V) \times [0, \infty)}$ is admissible.
\end{itemize}
\item There exists a map $v \colon S(G) \rightarrow \mathbb{R}^{2}$ such that
\begin{itemize}
\item for all $x \in S(G)$, we have $v(x) \notin dG_{x}(T_{x}V) \subset T_{G(x)} \mathbb{R}^{2} = \mathbb{R}^{2}$, and
\item there is $\varepsilon' \in (0, \varepsilon)$ such that $\sigma_{h}(y) \cdot v(x) \in \{0\} \times (0, \infty)$ for all $x = (s, y) \in S(G) \cap ([0, \varepsilon') \times \partial V) = [0, \varepsilon') \times S(g)$.
\end{itemize}
\end{enumerate}
\end{proposition}

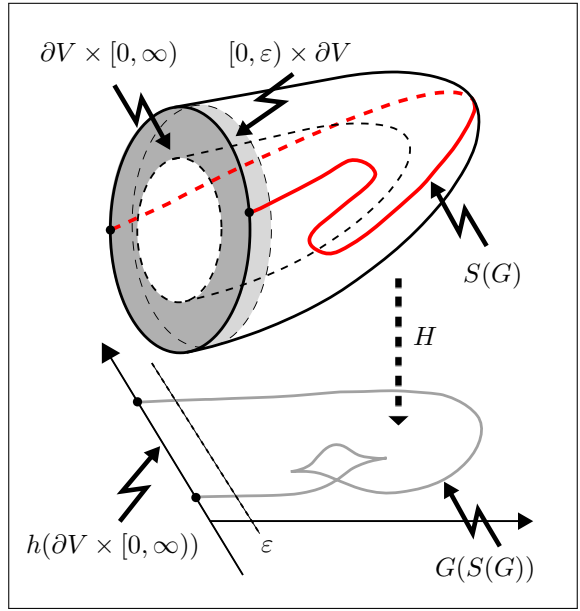
\begin{figure}[htbp]
\centering
\fbox{\begin{tikzpicture}
    \draw (0, 0) node {\includegraphics[height=0.6\textwidth]{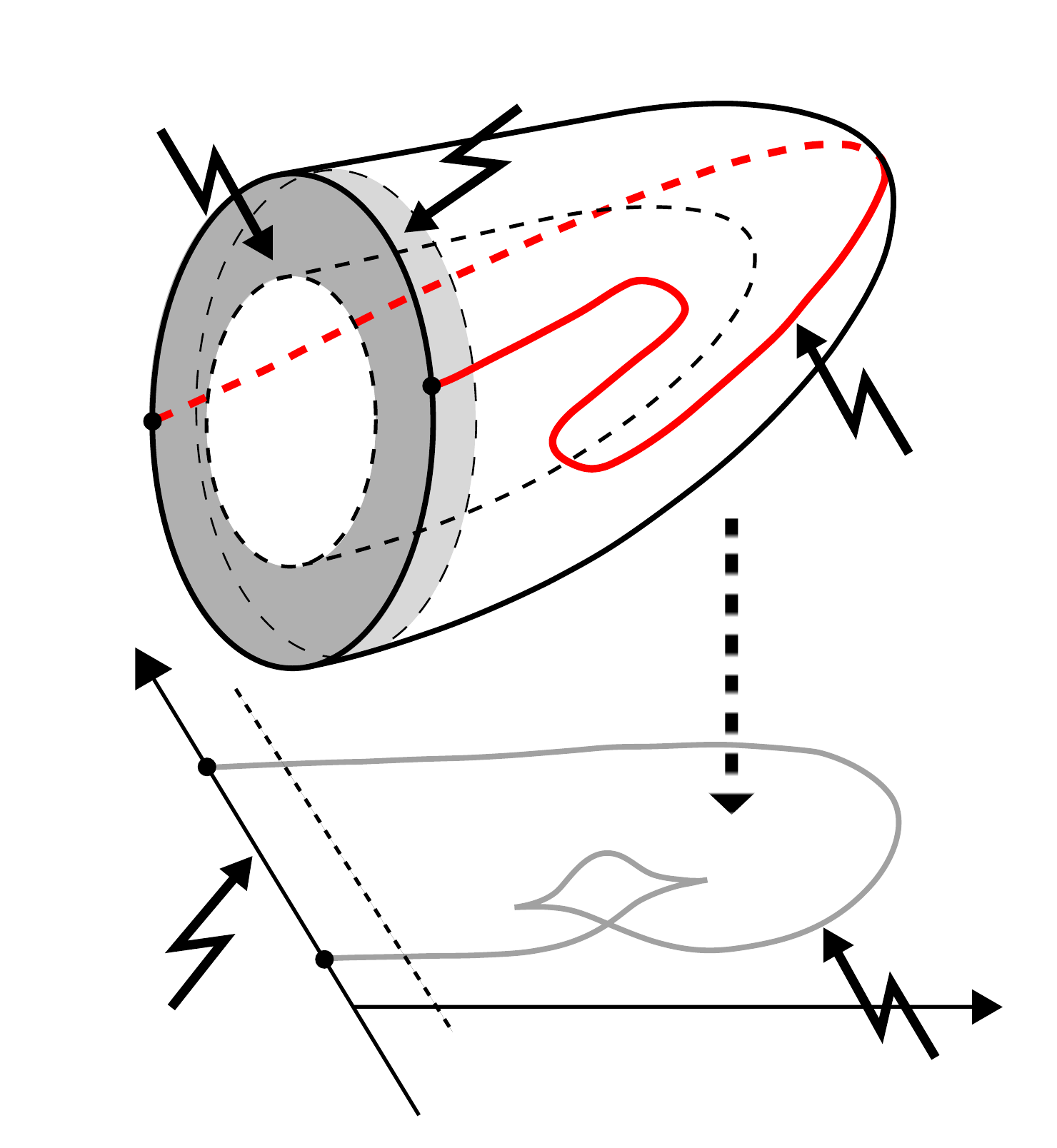}};
    \draw (1.7, -0.4) node {$H$};
    \draw (-2.5, 3.35) node {$\partial V \times [0, \infty)$};
    \draw (2.6, 0.4) node {$S(G)$};
    \draw (-0.1, 3.35) node {$[0, \varepsilon) \times \partial V$};
    \draw (-0.4, -3.2) node {$\varepsilon$};
    \draw (2.5, -3.5) node {$G(S(G))$};
    \draw (-2.45, -3.2) node {$h(\partial V \times [0, \infty))$};
\end{tikzpicture}}
\caption{On the manifold $V \times [0, \infty)$ with corners along $\partial V = \partial V \times \{0\}$ we consider a Morse function $g \colon \partial V \rightarrow \mathbb{R} = \{0\} \times \mathbb{R}$ with two extensions $h \colon \partial V \times [0, \infty) \rightarrow \mathbb{R}$ and $G \colon V \rightarrow \mathbb{R}^{2}$.
An extension $H \colon V \times [0, \infty) \rightarrow \mathbb{R}^{2}$ of $h \cup_{\partial V} G$ with the properties of \Cref{proposition extension of G}(i) exists if and only if there is a map $v \colon S(G) \rightarrow \mathbb{R}^{2}$ with the properties of \Cref{proposition extension of G}(ii).}
\label{figure proposition non-singular extension}
\end{figure}

\begin{proof}
(i) $\Rightarrow$ (ii).
Given $H \colon V \times [0, \infty) \rightarrow \mathbb{R}^{2}$, we consider the map
$$
v \colon S(G) \rightarrow \mathbb{R}^{2}, \quad v(x) = \partial_{t} H(x, t)|_{t = 0}.
$$
Let us check the claimed properties of $v$.
Given $x \in S(G)$, the real vector space $dG_{x}(T_{x} V)$ has dimension $1$ because $x$ is either a fold point or a cusp of $G$, while the real vector space
\begin{align*}
&d H_{(x, 0)}(T_{(x, 0)}(V \times [0, \infty))) \\
&= dG_{x}(T_{x} V) + dH_{(x, 0)}(0 \times T_{0} [0, \infty)) \\
&= dG_{x}(T_{x} V) + \mathbb{R} \cdot v(x)
\end{align*}
must be $2$-dimensional because $H$ is a submersion at $x = (x, 0) \in V \times [0, \infty)$.
Hence, $v(x) \notin dG_{x}(T_{x}V)$ for all $x \in S(G)$.
Moreover, for $x = (s, y) \in S(G) \cap ([0, \varepsilon') \times \partial V) = [0, \varepsilon') \times S(g)$ we have
$$
v(x) = \partial_{t} H(x, t)|_{t = 0} = (0, \partial_{t} h(y, t)|_{t = 0}).
$$
Since by definition of $\sigma_{h}$ the tangent vector
$$
\sigma_{h}(y) \cdot \partial_{t} h(y, t)|_{t = 0} = \sigma_{h}(y) \cdot dh_{(y, 0)}(\partial_{t}|_{(y, 0)}) \in T_{h(y, 0)} \mathbb{R} = \mathbb{R}
$$
points into the positive direction of the real axis, it follows that $\sigma_{h}(y) \cdot v(x) \in \{0\} \times (0, \infty)$.

(ii) $\Rightarrow$ (i).
Let $\Sigma \subset V \setminus \partial V$ denote the set of cusps of $G$.
For every $c \in \Sigma$ we choose a small open disk neighborhood $B_{c}$ of $c$ in $V \setminus ([0, \varepsilon) \times \partial V)$ such that $B_{c} \cap S(G)$ is connected, and define
$$
H_{c} \colon B_{c} \times [0, \infty) \rightarrow \mathbb{R}^{2}, \quad H_{c}(x, t) = G(x) + t \cdot v(c).
$$
Note that if $B_{c}$ is chosen sufficiently small, then $H_{c}$ is a submersion at every point of $B_{c} = B_{c} \times \{0\}$.
(Indeed, note that $v(c) \notin dG_{c}(T_{c}V)$ implies that $v(c) \notin dG_{x}(T_{x}V)$ for all $x \in S(G)$ that are sufficiently close to $c$.)
We may assume that $B_{c} \cap B_{c'} = \emptyset$ for $c \neq c'$ in $\Sigma$, and write $B_{\Sigma} = \bigsqcup_{c \in \Sigma} B_{c}$ and $H_{\Sigma} = \bigsqcup_{c \in \Sigma} H_{c}$.
In the following, we extend the $\Sigma \times [0, \infty)$-germ of the map $H_{\Sigma}$ and the $\{0\} \times \partial V \times [0, \infty)$-germ of the map
$$
H_{\varepsilon} = \operatorname{id}_{[0, \varepsilon)} \times h \colon [0, \varepsilon) \times \partial V \times [0, \infty) \rightarrow [0, \varepsilon) \times \mathbb{R}
$$
to the desired map $H \colon V \times [0, \infty) \rightarrow \mathbb{R}^{2}$.
For this purpose, we fix a compact neighborhood $\Sigma \subset K \subset B_{\Sigma}$,
and extend $v|_{S(G) \setminus (K \sqcup \partial V)}$ to a map
$$
\overline{v} \colon V \setminus (K \sqcup \partial V) \rightarrow \mathbb{R}^{2}.
$$
(This is possible because $v|_{S(G) \setminus (K \sqcup \partial V)}$ is defined on the $1$-dimensional submanifold $S(G) \setminus (K \sqcup \partial V) \subset V \setminus (K \sqcup \partial V)$.)
We use $\overline{v}$ to define the map
$$
H_{v} \colon (V \setminus (K \sqcup \partial V)) \times [0, \infty) \rightarrow \mathbb{R}^{2}, \qquad H_{v}(x, t) = G(x) + t \cdot \overline{v}(x).
$$
Note that $H_{v}$ is a submersion at every point of $V \setminus (K \sqcup \partial V) = (V \setminus (K \sqcup \partial V)) \times \{0\}$ because $\overline{v}(x) = v(x) \notin dG_{x}(T_{x}V)$ for all $x \in S(G) \setminus (K \sqcup \partial V)$.

Let $(\alpha, \beta)$ be a partition of unity subordinate to the open cover $V = U_{\alpha} \cup U_{\beta}$ given by
$$
U_{\alpha} = ([0, \varepsilon) \times \partial V) \sqcup B_{\Sigma}, \qquad U_{\beta} = V \setminus (K \sqcup \partial V).
$$
The pair $(\alpha, \beta)$ consists of functions $\alpha, \beta \colon V \rightarrow [0, 1]$ such that $\alpha(x) + \beta(x) = 1$ for all $x \in V$, and $\alpha|_{V \setminus K_{\alpha}} = 0$, $\beta|_{V \setminus K_{\beta}} = 0$ for some compact subsets $K_{\alpha} \subset U_{\alpha}$, $K_{\beta} \subset U_{\beta}$.
In particular, we have $\alpha|_{[0, \varepsilon'] \times \partial V} = 1$ for some $\varepsilon' \in (0, \varepsilon)$, $\alpha|_{K} = 1$, and $\beta|_{V \setminus (([0, \varepsilon) \times \partial V) \sqcup B_{\Sigma})} = 1$.
In the following, we use the partition of unity $(\alpha, \beta)$ to glue the maps $H_{\varepsilon} \sqcup H_{\Sigma}$ defined on $U_{\alpha} \times [0, \infty)$ and $H_{v}$ defined on $U_{\beta} \times [0, \infty)$ to obtain the desired map $H$.
Using the obvious extensions by zero, the maps $(x, t) \mapsto \alpha(x) \cdot (H_{\varepsilon} \sqcup H_{\Sigma})(x, t)$ and $(x, t) \mapsto \beta(x) \cdot H_{v}(x, t)$ extend to smooth maps on $V \times [0, \infty)$, and we define the map
$$
H \colon V \times [0, \infty) \rightarrow \mathbb{R}^{2}, \qquad H(x, t) = \alpha(x) \cdot (H_{\varepsilon} \sqcup H_{\Sigma})(x, t) + \beta(x) \cdot H_{v}(x, t).
$$
By construction, we have $H|_{V \times \{0\}} = G$ and $H|_{[0, \varepsilon') \times \partial V \times [0, \infty)} = H_{\varepsilon}|_{[0, \varepsilon') \times \partial V \times [0, \infty)} = \operatorname{id}_{[0, \varepsilon')} \times h$.
It remains to show that $H$ is a submersion at every point of $V = V \times \{0\}$.
Since $H|_{V \times \{0\}} = G$, it suffices to show that $H$ is a submersion at every point of $S(G)$.
This is clear for every cusp $c \in \Sigma$ of $G$ because $H|_{K \times [0, \infty)} = H_{\Sigma}|_{K \times [0, \infty)}$.
For every $x \in S(G) \setminus \Sigma$ we have to show that $\partial_{t} H(x, 0) \notin d G_{x}(T_{x}V)$.
In fact, if $x \in S(G) \setminus K_{\alpha}$, then $\beta(x) = 1$ implies that $\partial_{t}H(x, 0) = \overline{v}(x) = v(x) \notin d G_{x}(T_{x}V)$ by assumption on $v$.
In the remaining case that $x \in (S(G) \setminus \Sigma) \cap U_{\alpha}$ we prove our claim by showing that
$$
\partial_{t}H(x, 0) = \alpha(x) \cdot \partial_{t}(H_{\varepsilon} \sqcup H_{\Sigma})(x, 0) + \beta(x) \cdot v(x)
$$
is the convex combination of two (non-zero) vectors $\partial_{t}(H_{\varepsilon} \sqcup H_{\Sigma})(x, 0)$ and $v(x)$ that point to the same side of the line $d G_{x}(T_{x}V) \subset T_{G(x)}\mathbb{R}^{2} = \mathbb{R}^{2}$.
Indeed, if $x = (s, y) \in S(G) \cap ([0, \varepsilon) \times \partial V) = [0, \varepsilon) \times S(g)$, then
$$
\partial_{t}(H_{\varepsilon} \sqcup H_{\Sigma})(x, 0) = \partial_{t}H_{\varepsilon}(s, y, t)|_{t=0} = (0, \partial_{t} h(y, t)|_{t =0}).
$$
Using the definition of $\sigma_{h}$, we conclude that $\sigma_{h}(y) \cdot \partial_{t}(H_{\varepsilon} \sqcup H_{\Sigma})(x, 0) \in \{0\} \times (0, \infty)$.
Hence, it follows from the assumption $\sigma_{h}(y) \cdot v(x) \in \{0\} \times (0, \infty)$ that $v(x)$ and $\partial_{t}(H_{\varepsilon} \sqcup H_{\Sigma})(x, 0)$ point to the same side of the line $d G_{x}(T_{x}V) = \mathbb{R} \times \{0\} \subset T_{G(x)}\mathbb{R}^{2} = \mathbb{R}^{2}$.
Finally, if $x \in S(G) \cap (B_{c} \setminus \{c\})$ for some $c \in \Sigma$, then
$$
\partial_{t}(H_{\varepsilon} \sqcup H_{\Sigma})(x, 0) = \partial_{t}H_{c}(x, 0) = v(c).
$$
As the line $dG_{z}(T_{z}V) \subset T_{G(z)}\mathbb{R}^{2} = \mathbb{R}^{2}$ depends continuously on $z \in S(G)$ (compare \Cref{figure cusp construction}), it follows from $v(z), v(c) \notin dG_{z}(T_{z}V)$ for all $z \in S(G) \cap B_{c}$, and $v(z) \rightarrow v(c)$ for $z \rightarrow c$ that $v(x)$ and $v(c)$ point to the same side of the line $dG_{x}(T_{x}V)$.
\end{proof}

Condition (ii) of \Cref{proposition extension of G} can be related to the distribution of cusps of $G$ on the components of $S(G)$ as follows.

\begin{proposition}\label{proposition cusps and vector fields}
Let $V^{n}$ be a compact manifold possibly with boundary of dimension $n \geq 2$.
Suppose that $g \colon \partial V \rightarrow \mathbb{R}$ is a Morse function.
Let $G \colon V \rightarrow \mathbb{R}^{2}$ be a map such that $G|_{V \setminus \partial V}$ is generic, and $G|_{[0, \varepsilon) \times \partial V} = \operatorname{id}_{[0, \varepsilon)} \times g$ in some collar neighborhood $[0, \varepsilon) \times \partial V \subset V$ of $\partial V \subset V$.
Then, for any map $\sigma \colon S(g) \rightarrow \{\pm 1\}$ the following statements are equivalent:
\begin{enumerate}[(i)]
\item There is a map $v \colon S(G) \rightarrow \mathbb{R}^{2}$ such that
\begin{itemize}
\item for all $x \in S(G)$, we have $v(x) \notin dG_{x}(T_{x}V) \subset T_{G(x)} \mathbb{R}^{2} = \mathbb{R}^{2}$, and
\item there is $\varepsilon' \in (0, \varepsilon)$ such that $\sigma(y) \cdot v(x) \in \{0\} \times (0, \infty)$ for all $x = (s, y) \in S(G) \cap ([0, \varepsilon') \times \partial V) = [0, \varepsilon') \times S(g)$.
\end{itemize}
\item Every component $S$ of $S(G)$ satisfies the following.
If $\partial S = \emptyset$, then the number of cusps of $G$ that lie on $S$ is even.
If $\partial S \neq \emptyset$, then the number of cusps of $G$ that lie on $S$ is even if and only if $\sigma(x_{0}) \neq \sigma(x_{1})$, where $x_{0}, x_{1} \in S(g) = S(G) \cap \partial V$ denote the two endpoints of $S$.
\end{enumerate}
\end{proposition}

\begin{proof}
(i) $\Rightarrow$ (ii).
Fix an orientation $\mathfrak{o}$ of $S(G)$.
(In other words, $\mathfrak{o}$ is a non-singular vector field along $S(G) \subset V$.)
In the following, we consider the function $\delta \colon S(G) \rightarrow \mathbb{R}$ which assigns to $x \in S(G)$ the determinant of the $2 \times 2$-matrix $(dG_{x}(\mathfrak{o}_{x}), v(x))$.
Let $\Sigma$ denote the set of cusps of $G$.
\begin{itemize}
\item We have $\delta^{-1}(0) = \Sigma$. In fact, if $c$ is a cusp of $G$, then $dG_{c}(\mathfrak{o}_{c}) = 0$, and thus $\delta(c) = 0$.
On the other hand, if $x$ is a fold point of $G$, then $v(x) \notin dG_{x}(T_{x}V) = \mathbb{R} \cdot dG_{x}(\mathfrak{o}_{x}) \cong \mathbb{R}$, and thus $\delta(x) \neq 0$.

\item The sign of $\delta$ changes when passing through a cusp $c$ of $G$ while following $S(G)$ along the orientation $\mathfrak{o}$.
In fact, if $\phi \colon (-1, 1) \rightarrow S(G)$ is an orientation preserving embedding such that $\phi(0) = c$, then
$$
\operatorname{lim}_{t \nearrow 0} \frac{dG_{\phi(s)}(\mathfrak{o}_{\phi(s)})}{||dG_{\phi(s)}(\mathfrak{o}_{\phi(s)})||} = - \operatorname{lim}_{t \searrow 0} \frac{dG_{\phi(s)}(\mathfrak{o}_{\phi(s)})}{||dG_{\phi(s)}(\mathfrak{o}_{\phi(s)})||} \quad \in dG_{c}(T_{c}V),
$$
as indicated in \Cref{figure cusp construction}.
On the other hand, \Cref{figure cusp construction} shows that $v(\phi(s))$ is continuous in $s$, and we have $v(c) \notin dG_{c}(T_{c}V)$ by assumption.
Hence, it follows as claimed that the sign of the determinant of $(dG_{x}(\mathfrak{o}_{x}), v(x))$ changes when $x \in S(G)$ passes through the cusp $c$ of $G$.
\end{itemize}

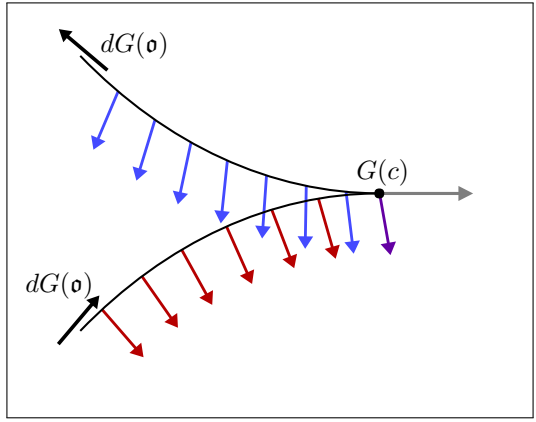
\begin{figure}[htbp]
\centering
\fbox{\begin{tikzpicture}
    \draw (0, 0) node {\includegraphics[height=0.4\textwidth]{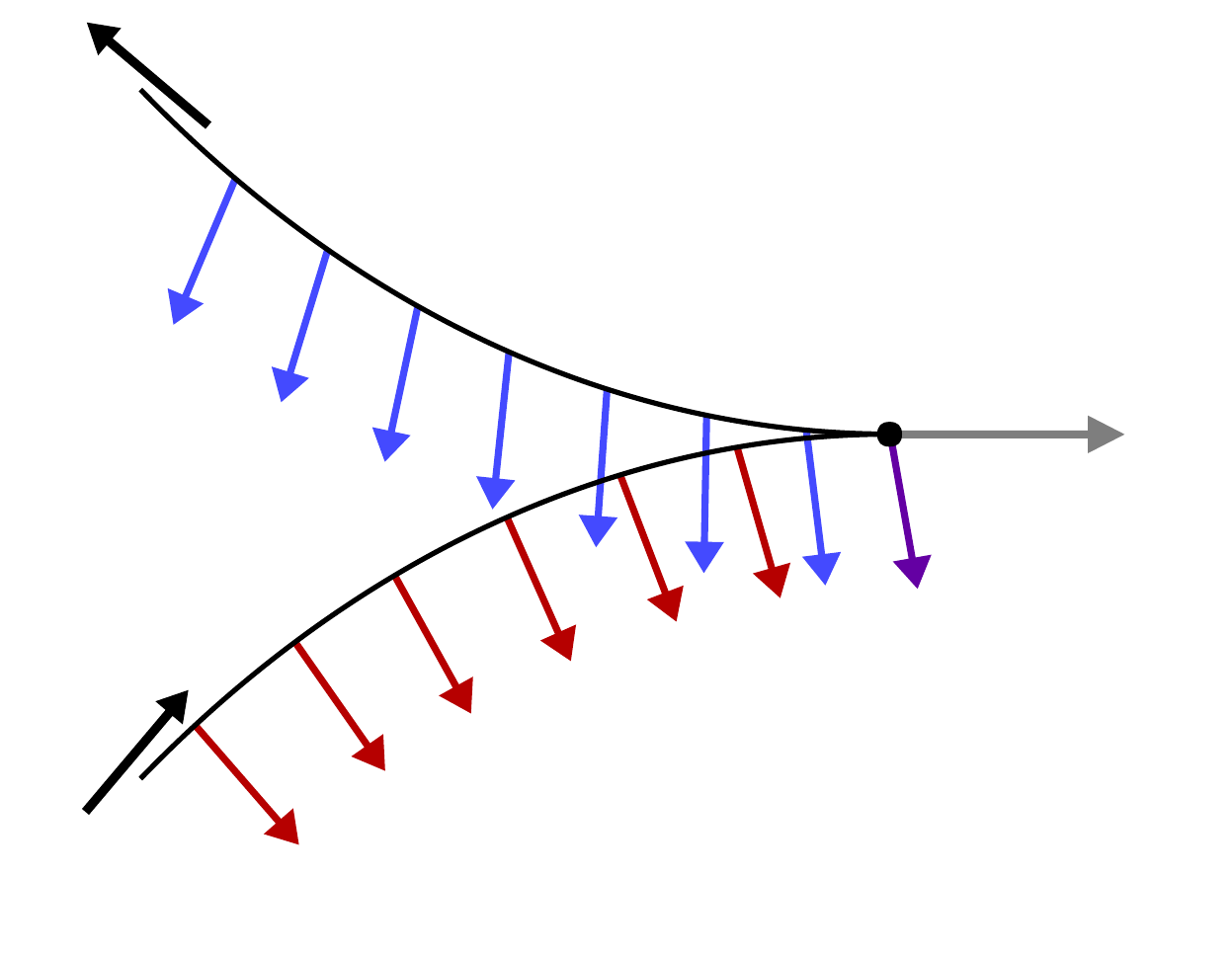}};
    \draw (1.5, 0.5) node {$G(c)$};
    \draw (-2.8, -1.0) node {$dG(\mathfrak{o})$};
    \draw (-1.8, 2.2) node {$dG(\mathfrak{o})$};
\end{tikzpicture}}
\caption{Illustration of the vector $v(z) \in T_{G(z)}\mathbb{R}^{2}$ for $z \in S(G)$ near a cusp $c$ of $G$.
By assumption on $v$ the vector $v(c)$ (purple arrow) does not lie in $dG_{c}(T_{c}V)$ (spanned by the grey arrow).
In the situation of this figure, when walking along $S(G)$ in the direction determined by the orientation $\mathfrak{o}$, $v(z)$ points to the right side of $dG(\mathfrak{o})$ before passing through the cusp $c$ (red arrows), and to the left side of $dG(\mathfrak{o})$ after passing through the cusp $c$ (blue arrows).}
\label{figure cusp construction}
\end{figure}

Now let $S$ be a component of $S(G)$.
If $\partial S = \emptyset$ (that is, if $S$ is diffeomorphic to the circle), then the above properties of $\delta$ imply that the number of cusps of $G$ that lie on $S$ is even.
If, however, $\partial S \neq \emptyset$ (that is, if $S$ is diffeomorphic to the interval $[0, 1]$), then the number of cusps of $G$ that lie on $S$ is either even or odd according to whether $\delta(x_{0}) \cdot \delta(x_{1}) > 0$ or $\delta(x_{0}) \cdot \delta(x_{1}) < 0$ at the two endpoints $x_{0}$ and $x_{1}$ of $S$.
Finally, we observe that $\delta(x_{0}) \cdot \delta(x_{1}) > 0$ holds if and only if $\sigma(x_{0}) \neq \sigma(x_{1})$.
(In fact, it follows from $dG_{x_{0}}(\mathfrak{o}_{x_{0}}) = - dG_{x_{1}}(\mathfrak{o}_{x_{1}}) \in \mathbb{R} \times \{0\} \subset \mathbb{R}^{2}$ that $\delta(x_{0}) \cdot \delta(x_{1}) >0$ holds if and only if $v(x_{0}), v(x_{1}) \in \{0\} \times \mathbb{R}$ point to different sides of the line $\mathbb{R} \times \{0\} \subset \mathbb{R}^{2}$.
But by the properties of $v$, this means that $\sigma(x_{0}) \neq \sigma(x_{1})$.)

(ii) $\Rightarrow$ (i).
Given an embedding $\psi \colon [0, 1] \rightarrow S(G)$, it is not hard to define a map $v_{\psi} \colon [0, 1] \rightarrow \mathbb{R}^{2}$ such that $v_{\psi}(s) \notin dG_{\psi(s)}(T_{\psi(s)}V)$ for all $s \in [0, 1]$.
Note that it is apparent from \Cref{figure cusp construction} how to define $v_{\psi}$ near cusps of $G$.

Fix a component $S$ of $S(G)$.
If $\partial S = \emptyset$, then we consider the map $v_{\psi_{S}} \colon [0, 1] \rightarrow \mathbb{R}^{2}$ obtained from an embedding $\psi_{S} \colon [0, 1] \rightarrow S$ such that $S \setminus \psi_{S}((0, 1))$ contains no cusps of $G$.
Since the number of cusps of $G$ that lie on $S$ is even by assumption, $v_{\psi_{S}}$ can clearly be extended to a map $v_{S} \colon S \rightarrow \mathbb{R}^{2}$ such that $v_{S}(x) \notin dG_{x}(T_{x}V)$ for all $x \in S$.
If $\partial S \neq \emptyset$, then $S$ is diffeomorphic to the interval $[0, 1]$, so that we can choose a map $v_{S} \colon S \rightarrow \mathbb{R}^{2}$ such that $v_{S}(x) \notin dG_{x}(T_{x}V)$ for all $x \in S$, and $v(x) \in \{0\} \times \mathbb{R}$ for all $x \in S(G) \cap ([0, \varepsilon') \times \partial V) = [0, \varepsilon') \times S(g)$ for some $\varepsilon' \in (0, \varepsilon)$.
Let $x_{0}$ and $x_{1}$ denote the two endpoints of $S$.
By replacing $v_{S}$ with $-v_{S}$ if necessary, we may assume that $\sigma(x_{0}) \cdot v_{S}(x_{0}) \in \{0\} \times (0, \infty)$.
In remains to show that $\sigma(x_{1}) \cdot v_{S}(x_{1}) \in \{0\} \times (0, \infty)$ as well.
For this purpose, we consider the function $\delta_{S} \colon S \rightarrow \mathbb{R}$ which assigns to $x \in S$ the determinant of the $2 \times 2$-matrix $(dG_{x}(\mathfrak{o}_{x}), v_{S}(x))$, where $\mathfrak{o}$ denotes an orientation of $S$.
Just as in the proof of the implication (i) $\Rightarrow$ (ii), it follows that $\delta_{S}^{-1}(0)$ is the set of cusps of $G$ on $S$, and the sign of $\delta_{S}$ changes when passing through a cusp $c$ of $G$ while following $S$ along the orientation $\mathfrak{o}$.
Hence, by construction, we have $\delta(x_{0}) \cdot \delta(x_{1}) > 0$ if and only if the number of cusps of $G$ that lie on $S$ is even, which is by (ii) furthermore equivalent to $\sigma(x_{0}) \neq \sigma(x_{1})$.
As before, note that $\delta(x_{0}) \cdot \delta(x_{1}) >0$ holds if and only if $v(x_{0}), v(x_{1}) \in \{0\} \times \mathbb{R}$ point to different sides of the line $\mathbb{R} \times \{0\} \subset \mathbb{R}^{2}$, and the claim $\sigma(x_{1}) \cdot v_{S}(x_{1}) \in \{0\} \times (0, \infty)$ follows.

Finally, we define $v$ by $v|_{S} = v_{S}$ for every component $S$ of $S(G)$.
\end{proof}

\section{Cusps and Euler characteristic}\label{Creation and elimination of cusps}
For a Morse function $g \colon P^{n-1} \rightarrow \mathbb{R}$ defined on a closed $(n-1)$-manifold, and a map $\sigma \colon S(g) \rightarrow \{\pm 1\}$, we shall use in this section the notation
\begin{align}\label{abstract sign Euler characteristic}
\chi_{+}(g; \sigma) = \sum_{i=0}^{n-1} (-1)^{i} \cdot \# S_{+}^{i}(g; \sigma),
\end{align}
where $S_{+}^{i}(g; \sigma) = S^{i}(g) \cap S_{+}(g; \sigma)$, and $S_{+}(g; \sigma) \subset S(g)$ denotes the subset of those critical points $x$ of $g$ for which $\sigma(x) = +1$.
Note that if $h \colon \partial V \times [0, \infty) \rightarrow \mathbb{R}$ is an admissible map in the sense of \Cref{Non-singular extensions} such that $h|_{\partial V \times \{0\}} = g$, then we have $\chi_{+}(g; \sigma_{h}) = \chi_{+}[h]$ as defined in \Cref{signed Euler characteristic}.
For later reference, we also point out that, using \Cref{euler characteristic},
\begin{align}\label{helpful expression}
\frac{\chi(P)}{2} - \chi_{+}(g; \sigma) = - \frac{1}{2} \sum_{i=0}^{n-1} (-1)^{i} \cdot \sum_{x \in S^{i}(g)} \sigma(x).
\end{align}

The purpose of this section is to express the existence of a map $G \colon V^{n} \rightarrow \mathbb{R}^{2}$ with the properties of \Cref{proposition cusps and vector fields} in terms of $\chi_{+}(g; \sigma)$.
The result is provided by \Cref{proposition condition n even} and \Cref{proposition condition n odd} depending on whether $n$ is even or odd.

\begin{lemma}\label{cusp invariant}
Let $G \colon Y \rightarrow \mathbb{R}^{2}$ be a map defined on a compact manifold possibly with boundary such that $G|_{Y \setminus \partial Y}$ is generic, and there is a Morse function $g \colon \partial Y \rightarrow \mathbb{R}$ such that $G|_{[0, \varepsilon) \times \partial Y} = \operatorname{id}_{[0, \varepsilon)} \times g$ in some collar neighborhood $[0, \varepsilon) \times \partial Y \subset Y$ of $\partial Y \subset Y$.
Then, $g$ has an even number of critical points, and  the number of cusps of $G$ has the same parity as $\chi(Y) + \frac{1}{2} \# S(g)$.
\end{lemma}

\begin{proof}
Without loss of generality, we may consider $G$ as a map $G \colon Y \rightarrow [0, 1] \times \mathbb{R}$ such that $G^{-1}(\{0\} \times \mathbb{R}) = \partial Y$ and $G^{-1}(\{1\} \times \mathbb{R}) = \emptyset$ (compare \Cref{remark target plane}).
Let $\Sigma$ denote the set of cusps of $G$.
After perturbing the generic map $G|_{Y \setminus \partial Y}$ on a compact subset, we may assume without loss of generality that the composition $\tau = \pi \circ G \colon Y \rightarrow [0, 1]$ of $G$ with the projection $\pi \colon [0, 1] \times \mathbb{R} \rightarrow [0, 1]$ to the first factor has the following properties:
\begin{itemize}
\item the function $\tau \colon Y \rightarrow [0, 1]$ satisfies $\tau^{-1}(0) = \partial Y$, $\tau^{-1}(1) = \emptyset$, and every critical point of $\tau$ is non-degenerate, and is also a fold point of $G$, and
\item the map $\tau$ restricts to a function $\tau_{S} \colon S(G) \rightarrow [0, 1]$ such that $\tau_{S}^{-1}(0) = S(g)$, $\tau_{S}^{-1}(1) = \emptyset$, the critical points of $\tau_{S}$ are all non-degenerate, and $S(\tau_{S}) = S(\tau) \sqcup \Sigma$.
\end{itemize}

Recall the fact from Morse theory that for a compact manifold possibly with boundary $W$ and a function $f \colon W \rightarrow [0, 1]$ with only non-degenerate critical points such that $\partial W = f^{-1}(0) \sqcup f^{-1}(1)$ and $\partial W \cap S(f) = \emptyset$, we have $\# S(f) \equiv \chi(W) + \chi(f^{-1}(0)) \; (\operatorname{mod} 2)$.
We apply this fact to the functions $g \colon \partial Y \rightarrow \mathbb{R}$, $\tau \colon Y \rightarrow [0, 1]$, and $\tau_{S} \colon S(G) \rightarrow [0, 1]$ to obtain the following equations for their numbers of critical points:
\begin{enumerate}[(1)]
\item $\# S(g) \equiv \chi(\partial Y) \; (\operatorname{mod} 2)$,
\item $\# S(\tau) \equiv \chi(Y) + \chi(\partial Y) \; (\operatorname{mod} 2)$, and
\item $\# S(\tau_{S}) \equiv \chi(S(G)) + \chi(S(g)) \; (\operatorname{mod} 2)$.
\end{enumerate}
Note that $\# S(\tau_{S}) = \# \Sigma + \# S(\tau)$ by the properties of $\tau_{S}$.
Moreover, note that $\chi(S(G)) = \frac{1}{2} \# S(g)$ and $\chi(S(g)) = \# S(g)$.
Hence, (3) implies
$$
\# \Sigma + \frac{1}{2} \# S(g) \equiv \# S(\tau) + \# S(g) \; (\operatorname{mod} 2).
$$
Finally, using (1) and (2), we conclude that $\# S(\tau) + \# S(g) \equiv \chi(Y) \; (\operatorname{mod} 2)$.
\end{proof}

\begin{proposition}\label{proposition condition n even}
Let $V^{n}$ be a connected compact manifold possibly with boundary of dimension $n \geq 2$.
Fix a Morse function $g \colon \partial V \rightarrow \mathbb{R}$ and a map $\sigma \colon S(g) \rightarrow \{\pm 1\}$.
If $n$ is even, then the following statements are equivalent:
\begin{enumerate}[(i)]
\item There exists a map $G \colon V \rightarrow \mathbb{R}^{2}$ such that $G|_{V \setminus \partial V}$ is generic, $G|_{[0, \varepsilon) \times \partial V} = \operatorname{id}_{[0, \varepsilon)} \times g$ in some collar neighborhood $[0, \varepsilon) \times \partial V \subset V$ of $\partial V \subset V$, and $G$ satisfies condition (ii) of \Cref{proposition cusps and vector fields}.
\item $\chi(V) \equiv \chi_{+}(g; \sigma) \; (\operatorname{mod} 2)$ (see \Cref{abstract sign Euler characteristic}).
\end{enumerate}
\end{proposition}

\begin{proof}
Consider a map $G \colon V \rightarrow \mathbb{R}^{2}$ such that $G|_{V \setminus \partial V}$ is generic, and $G|_{[0, \varepsilon) \times \partial V} = \operatorname{id}_{[0, \varepsilon)} \times g$ in some collar neighborhood $[0, \varepsilon) \times \partial V \subset V$ of $\partial V \subset V$.
Let $\Sigma \subset S(G)$ denote the set of cusps of $G$.
Note that condition (ii) of \Cref{proposition cusps and vector fields} is equivalent to the requirement that every component $S$ of $S(G)$ satisfies
\begin{align}\label{equation summand}
\# (\Sigma \cap S) + \frac{1}{2} \sum_{x \in \partial S} \sigma(x) \equiv 0 \quad (\operatorname{mod} 2).
\end{align}
Writing $\mathcal{S}(G)$ for the set of components of $S(G)$, we conclude from \Cref{cusp invariant} that
\begin{align}\label{modulo two formula}
\chi(V) - \chi_{+}(g; \sigma)
\equiv \sum_{S \in \mathcal{S}(G)} \Bigg[\# (\Sigma \cap S) + \frac{1}{2} \sum_{x \in \partial S} \sigma(x)\Bigg] \quad (\operatorname{mod} 2).
\end{align}

(i) $\Rightarrow$ (ii).
If we choose $G$ to satisfy condition (ii) of \Cref{proposition cusps and vector fields}, then every component $S$ of $S(G)$ satisfies \Cref{equation summand}, and statement (ii) follows immediately from \Cref{modulo two formula}.

(ii) $\Rightarrow$ (i).
Choose a collar neighborhood $[0, \infty) \times \partial V \subset V$ of $\partial V \subset V$.
By \Cref{corollary generic maps}, we may extend the $\{0\} \times \partial V$-germ of $\operatorname{id}_{[0, \infty)} \times g$ to a map $G_{0} \colon V \rightarrow \mathbb{R}^{2}$ such that $G_{0}|_{V \setminus \partial V}$ is generic, and $G_{0}|_{[0, \varepsilon) \times \partial V} = \operatorname{id}_{[0, \varepsilon)} \times g$ in some collar neighborhood $[0, \varepsilon) \times \partial V \subset V$ of $\partial V \subset V$.
Assuming that statement (ii) holds, we proceed in two steps as follows to modify $G_{0}$ on a compact subset of $V \setminus \partial V$ in such a way that the modified map $G \colon V \rightarrow \mathbb{R}^{2}$ has the desired properties.

\textit{Step 1.} In the first step, we modify $G_{0}$ on a compact subset of $V \setminus \partial V$ in such a way that for the modified map $G_{1} \colon V \rightarrow \mathbb{R}^{2}$, the restriction $G_{1}|_{V \setminus \partial V}$ is generic, and every component $S$ of $S(G_{1})$ which is diffeomorphic to $[0, 1]$ satisfies \Cref{equation summand}.
For this purpose, let $\mathcal{S}_{[0, 1]}$ denote the set of components of $S(G_{0})$ which are diffeomorphic to $[0, 1]$, but do not satisfy \Cref{equation summand}.
For every $S \in \mathcal{S}_{[0, 1]}$ we choose an open neighborhood $X_{S} \subset V$ of $S$ such that $X_{S} \cap X_{S'} = \emptyset$ whenever $S \neq S'$ in $\mathcal{S}_{[0, 1]}$.
Noting that $n \geq 2$ is even, we apply \Cref{lemma cusp number change in even dimensions} for every $S \in \mathcal{S}_{[0, 1]}$ to the manifold $X = X_{S} \setminus \partial V$ (without boundary), to the generic map $G_{0}|_{X_{S} \setminus \partial V}$, to the singular component $C_{0} = S \setminus \partial V$, and to some open subset $U \subset X$ with compact closure in $X$ such that $U \cap S(G_{0}) = U \cap C_{0} \neq \emptyset$.
As a result, we obtain for every $S \in \mathcal{S}_{[0, 1]}$ a generic map $G_{S} \colon X_{S} \setminus \partial V \rightarrow \mathbb{R}^{2}$, a compact subset $K_{S} \subset X_{S} \setminus \partial V$ such that $G_{S}|_{X_{S} \setminus (K_{S} \cup \partial V)} = G_{0}|_{X_{S} \setminus (K_{S} \cup \partial V)}$, and a component $C_{S}$ of $S(G_{S})$ with the properties that $C_{S} \setminus K_{S} = S \setminus K_{S}$, and that the number of cusps of $G_{S}$ lying on $C_{S}$ has not the same parity as the number of cusps of $G_{0}$ lying on $S$.
The map $G_{1} \colon V \rightarrow \mathbb{R}^{2}$ given by $G_{1}|_{V \setminus \bigcup_{S \in \mathcal{S}_{[0, 1]}} K_{S}} = G_{0}|_{V \setminus \bigcup_{S \in \mathcal{S}_{[0, 1]}} K_{S}}$ and $G_{1}|_{X_{S} \setminus \partial V} = G_{S}$ for every $S \in \mathcal{S}_{[0, 1]}$ then has the desired properties.

\textit{Step 2.}
In the second step, we modify $G_{1}$ on a compact subset of $V \setminus \partial V$ in such a way that for the modified map $G \colon V \rightarrow \mathbb{R}^{2}$, the restriction $G|_{V \setminus \partial V}$ is generic, and every component $S$ of $S(G)$ satisfies \Cref{equation summand}, that is, $G$ satisfies condition (ii) of \Cref{proposition cusps and vector fields}.
We recall from the previous step that \Cref{equation summand} is satisfied by every component $S$ of $S(G_{1})$ that is diffeomorphic to $[0, 1]$.
Consequently, using statement (ii) and \Cref{modulo two formula}, we conclude that there is an even number of cusps of $G_{1}$ lying on components of $S(G_{1})$ which are diffeomorphic to the circle.
For the desired modification of $G_{1}$ we distinguish between the following two cases.
\begin{itemize}
\item The case $n=2$:
We apply the method of singular patterns as developed in \cite{wra2}.
Let $W$ be the connected compact $2$-manifold obtained by removing from $V$ suitable small open disk neighborhoods of those cusps of $G_{1}$ which lie on components of $S(G_{1})$ that are diffeomorphic to $[0, 1]$.
The $\partial W$-germ of $\widetilde{G}_{1} = G_{1} \cup \operatorname{id}_{(-\varepsilon, 0]} \times g$ on $\widetilde{V} = V \cup (-\varepsilon, 0] \times \partial V$ and the singular set $S(G_{1}|_{W})$ then induce a \emph{singular pattern} $(f, \varphi)$ on $W$ in the sense of Definition 5.1 in \cite{wra2}.
That is, $f \colon (-\varepsilon, \varepsilon) \times \partial W \rightarrow \mathbb{R}^{2}$, $\varepsilon > 0$, is a fold map whose singular locus $S(f) \subset (-\varepsilon, \varepsilon) \times \partial W$ is transverse to $\{0\} \times \partial W$, and $\varphi$ is a partition of the finite set $S(f) \cap \{0\} \times \partial W$ into subsets of cardinality two.
Namely, we take $f$ to be the restriction of $\widetilde{G}_{1}$ to a tubular neighborhood $(-\varepsilon, \varepsilon) \times \partial W \subset \widetilde{V}$ of $\partial W = \{0\} \times \partial W$ in $\widetilde{V}$, and $\varphi$ to be the partition consisting of the sets of the form $\partial S$, where $S$ runs through the components of $S(G_{1}|_{W})$ that are diffeomorphic to $[0, 1]$.
Then, by construction, $G_{1}|_{W}$ is a realization of the singular pattern $(f, \varphi)$ on $W$ in the sense of Definition 5.2 in \cite{wra2}.
Furthermore, the number of cusps of $G_{1}|_{W}$ is even.
Hence, by Theorem 1.1 in \cite{wra2}, $G_{1}|_{W}$ can be modified on a compact subset of $W \setminus \partial W$ to a fold map $H \colon W \rightarrow \mathbb{R}^{2}$ that realizes the pattern $(f, \varphi)$ on $W$.
Finally, the map $G \colon V \rightarrow \mathbb{R}^{2}$ defined by $G|_{W} = H$ and $G|_{V \setminus W} = G_{1}|_{V \setminus W}$ is a modification of $G_{1}$ on a compact subset of $V \setminus \partial V$ in such a way that the restriction $G|_{V \setminus \partial V}$ is generic.
Moreover, it follows from the construction that every component $S$ of $S(G)$ satisfies \Cref{equation summand}.
(Note that the components of $S(G)$ which are diffeomorphic to the circle do not contain any cusps of $G$.)
Consequently, $G$ satisfies condition (ii) of \Cref{proposition cusps and vector fields}.

\item The case $n>2$:
Let $\mathcal{S}_{S^{1}}$ denote the set of components of $S(G_{1})$ which are diffeomorphic to the circle, but do not satisfy \Cref{equation summand}, which means that they contain an odd number of cusps of $G_{1}$.
According to the corollary of Section (3.2) in \cite[p. 275]{lev}, each component $S \in \mathcal{S}_{S^{1}}$ contains at least one cusp of $G_{1}$ absolute index $(n-2)/2$, say $c_{S} \in S$.
As the cardinality of $\mathcal{S}_{S^{1}}$ is even by assumption, we may fix a partition $\mathcal{P}$ of $\mathcal{S}_{S^{1}}$ into sets of cardinality two.
Note that for every element $\{S, S'\} \in \mathcal{P}$ the pair of cusps $(c_{S}, c_{S'})$ is a matching pair in the sense of \Cref{definition matching pair}.
Since the manifold $V^{n}$ is connected and of dimension $n > 2$, the pair $(c_{S}, c_{S'})$ is also removable, which means that we can apply \Cref{proposition elimination of cusps} to modify $G_{1}$ in a small neighborhood of some joining curve.
More precisely, we choose for every element $\mathfrak{s} = \{S, S'\} \in \mathcal{P}$ a joining curve $\lambda_{\mathfrak{s}} \colon [0, 1] \rightarrow V \setminus \partial V$ for $(c_{S}, c_{S'})$, and a small neighborhood $U_{\mathfrak{s}} \subset V \setminus \partial V$ of $S \cup \lambda_{\mathfrak{s}}([0, 1]) \cup S'$ such that $U_{\mathfrak{s}} \cap S(G_{1}) = S \sqcup S'$ for all $\mathfrak{s} \in \mathcal{P}$, and $U_{\mathfrak{s}} \cap U_{\mathfrak{s}'} = \emptyset$ for $\mathfrak{s} \neq \mathfrak{s}'$ in $\mathcal{P}$.
Then, \Cref{proposition elimination of cusps} yields for every $\mathfrak{s} \in \mathcal{P}$ a generic map $G_{\mathfrak{s}} \colon U_{\mathfrak{s}} \rightarrow \mathbb{R}^{2}$ such that $G_{\mathfrak{s}}|_{U_{\mathfrak{s}} \setminus K_{\mathfrak{s}}} = G_{1}|_{U_{\mathfrak{s}} \setminus K_{\mathfrak{s}}}$ for some compact subset $K_{\mathfrak{s}} \subset U_{\mathfrak{s}}$, and $S(G_{\mathfrak{s}})$ is connected and contains an even number of cusps.
Finally, the map $G \colon V \rightarrow \mathbb{R}^{2}$ defined by $G|_{V \setminus \bigcup_{\mathfrak{s} \in \mathcal{P}} K_{\mathfrak{s}}} = G_{1}|_{V \setminus \bigcup_{\mathfrak{s} \in \mathcal{P}} K_{\mathfrak{s}}}$ and $G|_{U_{\mathfrak{s}}} = G_{\mathfrak{s}}$ for every $\mathfrak{s} \in \mathcal{P}$ is a modification of $G_{1}$ on a compact subset of $V \setminus \partial V$ in such a way that the restriction $G|_{V \setminus \partial V}$ is generic.
Moreover, it follows from the construction that every component $S$ of $S(G)$ satisfies \Cref{equation summand}, that is, $G$ satisfies condition (ii) of \Cref{proposition cusps and vector fields}.
\end{itemize}
Since the map $G$ obtained from the above construction has the desired properties, the proof of \Cref{proposition condition n even} is complete.
\end{proof}

\begin{proposition}\label{proposition condition n odd}
Let $V^{n}$ be a connected compact manifold possibly with boundary of dimension $n \geq 2$.
Fix a Morse function $g \colon \partial V \rightarrow \mathbb{R}$ and a map $\sigma \colon S(g) \rightarrow \{\pm 1\}$.
If $n$ is odd, then the following statements are equivalent:
\begin{enumerate}[(i)]
\item There exists a map $G \colon V \rightarrow \mathbb{R}^{2}$ such that $G|_{V \setminus \partial V}$ is generic, $G|_{[0, \varepsilon) \times \partial V} = \operatorname{id}_{[0, \varepsilon)} \times g$ in some collar neighborhood $[0, \varepsilon) \times \partial V \subset V$ of $\partial V \subset V$, and $G$ satisfies condition (ii) of \Cref{proposition cusps and vector fields}.
\item $\frac{\chi(\partial V)}{2} = \chi_{+}(g; \sigma)$ (see \Cref{abstract sign Euler characteristic}).
\end{enumerate}
\end{proposition}

\begin{proof}
Consider a map $G \colon V \rightarrow \mathbb{R}^{2}$ such that $G|_{V \setminus \partial V}$ is generic, and $G|_{[0, \varepsilon) \times \partial V} = \operatorname{id}_{[0, \varepsilon)} \times g$ in some collar neighborhood $[0, \varepsilon) \times \partial V \subset V$ of $\partial V \subset V$.
Let $\Sigma \subset S(G)$ denote the set of cusps of $G$.
We observe that condition (ii) of \Cref{proposition cusps and vector fields} is equivalent to the requirement that every component $S$ of $S(G)$ satisfies
\begin{align}\label{equation summand odd dimension}
\sum_{x \in \partial S} (-1)^{\mu(x)} \cdot \sigma(x) = 0,
\end{align}
where $\mu(x)$ denotes the Morse index of the critical point $x \in \partial S(G) = S(g)$.
(In order to prove this claim, suppose first that $S$ is a component of $S(G)$ which is diffeomorphic to the circle.
Then, as $n$ is odd, it follows from \Cref{remark index} (see also the corollary of Section (3.2) in \cite[p. 275]{lev}) that $S$ contains an even number of cusps, that is, $S$ satisfies condition (ii) of \Cref{proposition cusps and vector fields}.
Moreover, \Cref{equation summand odd dimension} is clearly satisfied for $S$.
Thus, our claim holds for every component $S$ of $S(G)$ which is diffeomorphic to the circle.
Next, suppose that $S$ is a component of $S(G)$ which is diffeomorphic to the interval $[0, 1]$, and let $x_{0}$ and $x_{1}$ denote the endpoints of $S$.
Then, we observe that the number of cusps of $G$ that lie on $S$ is even if and only if the absolute indices $\operatorname{max}\{\mu(x_{0}), n-1-\mu(x_{0})\}$ of $x_{0}$ and $\operatorname{max}\{\mu(x_{1}), n-1-\mu(x_{1})\}$ of $x_{1}$ have the same parity, that is, $\mu(x_{0}) \equiv \mu(x_{1}) \; (\operatorname{mod} 2)$ (because $n$ is odd).
The claim follows immediately from this observation.)

Writing $\mathcal{S}(G)$ for the set of components of $S(G)$, we rewrite \Cref{helpful expression} as
\begin{align}\label{integral formula}
\frac{\chi(\partial V)}{2} - \chi_{+}(g; \sigma) = - \frac{1}{2} \sum_{S \in \mathcal{S}(G)}\Bigg[\sum_{x \in \partial S} (-1)^{\mu(x)} \cdot \sigma(x)\Bigg].
\end{align}

(i) $\Rightarrow$ (ii).
If we choose $G$ to satisfy condition (ii) of \Cref{proposition cusps and vector fields}, then every component $S$ of $S(G)$ satisfies \Cref{equation summand odd dimension}, and statement (ii) follows immediately from \Cref{integral formula}.

(ii) $\Rightarrow$ (i).
Choose a collar neighborhood $[0, \infty) \times \partial V \subset V$ of $\partial V \subset V$.
By \Cref{corollary generic maps}, we may extend the $\{0\} \times \partial V$-germ of $\operatorname{id}_{[0, \infty)} \times g$ to a map $G_{0} \colon V \rightarrow \mathbb{R}^{2}$ such that $G_{0}|_{V \setminus \partial V}$ is generic, and $G_{0}|_{[0, \varepsilon) \times \partial V} = \operatorname{id}_{[0, \varepsilon)} \times g$ in some collar neighborhood $[0, \varepsilon) \times \partial V \subset V$ of $\partial V \subset V$.
Assuming that statement (ii) holds, \Cref{integral formula} implies that we may choose a partition $\mathcal{P}$ of $S(g)$ such that for every element $\{x_{0}, x_{1}\} \in \mathcal{P}$ we have
$$
\sum_{x \in \{x_{0}, x_{1}\}} (-1)^{\mu(x)} \cdot \sigma(x) = 0.
$$
We proceed as follows to modify $G_{0}$ on a compact subset of $V \setminus \partial V$ in such a way that the modified map $G \colon V \rightarrow \mathbb{R}^{2}$ has the desired properties, namely that the restriction $G|_{V \setminus \partial V}$ is generic, and that every component $S$ of $S(G)$ satisfies \Cref{equation summand odd dimension}.
Let $\mathcal{P}_{\partial}$ denote the set of elements $\mathfrak{p} = \{x_{0}, x_{1}\} \in \mathcal{P}$ which are not of the form $\mathfrak{p} = \partial S$ for some component $S$ of $S(G_{0})$.
Since the manifold $V$ is connected and of dimension $n > 2$, we can choose for every element $\mathfrak{p} = \{x_{0}, x_{1}\} \in \mathcal{P}_{\partial}$ a connected open subset $U_{\mathfrak{p}} \subset V \setminus \partial V$ whose closure in $V \setminus \partial V$ is compact, and such that $U_{\mathfrak{p}} \cap S(G_{1}) \subset S_{0} \cup S_{1}$, and $U_{\mathfrak{p}} \cap S_{0}$ and $U_{\mathfrak{p}} \cap S_{1}$ are connected, where $S_{0}$ and $S_{1}$ denote the components of $S(G_{0})$ that contain $x_{0}$ and $x_{1}$, respectively.
Moreover, we may assume that $U_{\mathfrak{p}} \cap U_{\mathfrak{p}'} = \emptyset$ whenever $\mathfrak{p} \neq \mathfrak{p}'$ in $\mathcal{P}_{\partial}$.
We apply \Cref{lemma surgery of singular components} for every $\mathfrak{p} = \{x_{0}, x_{1}\} \in \mathcal{P}_{\partial}$ to the manifold $X = U_{\mathfrak{p}}$ (without boundary), to the generic map $G_{0}|_{U_{\mathfrak{p}}}$, to the components $C_{0}^{(0)} = U_{\mathfrak{p}} \cap S_{0}$ and $C_{0}^{(1)} = U_{\mathfrak{p}} \cap S_{1}$ of $S(G_{0}|_{X})$, to some connected open subset $U \subset X$ with compact closure in $X$ such that $U \cap C_{0}^{(0)} \neq \emptyset$ and $U \cap C_{0}^{(1)} \neq \emptyset$, and to some points $x^{(0)} \in C_{0}^{(0)} \setminus U$ and $x^{(1)} \in C_{0}^{(1)} \setminus U'$, where $x^{(0)}$ lies on the same component of $S_{0} \setminus U$ as $x_{0}$, and $x^{(1)}$ lies on the same component of $S_{1} \setminus U$ as $x_{1}$.
As a result, we obtain for every $\mathfrak{p} \in \mathcal{P}_{\partial}$ a generic map $G_{\mathfrak{p}} \colon U_{\mathfrak{p}} \rightarrow \mathbb{R}^{2}$ with the following properties.
We have $G_{\mathfrak{p}}|_{U_{\mathfrak{p}} \setminus K_{\mathfrak{p}}} = G_{0}|_{U_{\mathfrak{p}} \setminus K_{\mathfrak{p}}}$ for some compact subset $K_{\mathfrak{p}} \subset U_{\mathfrak{p}}$, and $S(G_{\mathfrak{p}})$ contains a component that contains the points $x^{(0)}$ and $x^{(1)}$.
Finally, the map $G \colon V \rightarrow \mathbb{R}^{2}$ defined by $G|_{V \setminus \bigcup_{\mathfrak{p} \in \mathcal{P}_{\partial}} K_{\mathfrak{p}}} = G_{0}|_{V \setminus \bigcup_{\mathfrak{p} \in \mathcal{P}_{\partial}} K_{\mathfrak{p}}}$ and $G|_{U_{\mathfrak{p}}} = G_{\mathfrak{p}}$ for every $\mathfrak{p} \in \mathcal{P}_{\partial}$ is a modification of $G_{0}$ on a compact subset of $V \setminus \partial V$ in such a way that the restriction $G|_{V \setminus \partial V}$ is generic.
Moreover, it follows from the construction that every $\mathfrak{p} \in \mathcal{P}_{\partial}$ is of the form $\mathfrak{p} = \partial S$ for some component $S$ of $S(G)$.
Hence, it follows that every component $S$ of $S(G)$ satisfies \Cref{equation summand odd dimension}.

This completes the proof of \Cref{proposition condition n odd}.
\end{proof}

\section{Proof of \Cref{main theorem}}\label{proof of main theorem}

In \Cref{proposition existence of homomorphism} and \Cref{proposition existence of homomorphism for n odd} we discuss the construction of the homomorphisms that appear in \Cref{main theorem}.
Afterwards, we show that they are surjective and injective.

\begin{lemma}\label{lemma euler characteristic}
\begin{enumerate}[(a)]
\item If $X$ is an odd-dimensional compact manifold possibly with boundary, then $\chi(X) = \frac{\chi(\partial X)}{2}$.
\item If $X_{0}$ and $X_{1}$ are compact even-dimensional manifolds with (possibly empty) common boundary $Y = \partial X_{0} = \partial X_{1}$ such that $X_{0} \cup_{Y} X_{1}$ is nullcobordant, then $\chi(X_{0}) \equiv \chi(X_{1}) \; (\operatorname{mod} 2)$.
\end{enumerate}
\end{lemma}

\begin{proof}
(a).
As the double $X \cup_{\partial X} X$ is a closed odd-dimensional manifold, we have
$$
0 = \chi(X \cup_{\partial X} X) = 2 \chi(X) - \chi(\partial X).
$$

(b).
As $Y$ is a closed odd-dimensional manifold, we have
$$
0 = \chi(Y) = \chi(X_{0}) + \chi(X_{1}) - \chi(X_{0} \cup_{Y} X_{1}),
$$
where $\chi(X_{0} \cup_{Y} X_{1})$ is even by part (a).
\end{proof}

\begin{proposition}\label{proposition existence of homomorphism}
If $n \geq 2$ is even, then there exists a homomorphism
\begin{align*}
\mathcal{C}_{n} \rightarrow \mathbb{Z}/2, \qquad [f \colon M \rightarrow \mathbb{R}] \mapsto \chi(M) - \chi_{+}[f] \; (\operatorname{mod} 2).
\end{align*}
\end{proposition}

\begin{proof}
In order to show that the desired map is well-defined, let us consider two Morse functions $f_{1} \colon M_{1} \rightarrow \mathbb{R}$ and $f_{2} \colon M_{2} \rightarrow \mathbb{R}$ defined on compact oriented $n$-manifolds possibly with boundary that represent the same element in $\mathcal{C}_{n}$.
Then, the Morse function $f = f_{1} \sqcup -f_{2}$ defined on $M = M_{1} \sqcup -M_{2}$ represents the element $[f_{1}] - [f_{2}] = 0 = [f_{\emptyset} \colon \emptyset \rightarrow \mathbb{R}] \in \mathcal{C}_{n}$.
By \Cref{definition cusp cobordism}, there exists an oriented cobordism $(W, V)$ from $M$ to $\emptyset$ (compare \Cref{figure injectivity}) together with a map $F \colon W \rightarrow [0, 1] \times \mathbb{R}$ such that $F^{-1}(\mathbb{R} \times \{0\}) = M$ and $F^{-1}(\mathbb{R} \times \{1\}) = \emptyset$, and the following properties hold:
\begin{enumerate}[(i)]
\item For some $\varepsilon > 0$ there exists a collar neighborhood (with corners) $[0, \varepsilon) \times M \subset W$ of $\{0\} \times M = M \subset W$ such that $F|_{[0, \varepsilon) \times M} = \operatorname{id}_{[0, \varepsilon)} \times f$.
\item The restriction $F|_{W \setminus M}$ is a submersion at every point of the boundary $V \setminus \partial V$ of $W \setminus M$.
\item The restrictions $F|_{W \setminus \partial W}$ and $F|_{V \setminus \partial V}$ are generic maps into the plane.
\end{enumerate}
Let us fix a collar neighborhood (with corners) $V \times [0, \infty) \subset W$ of $V \subset W$.
We apply implication (i) $\Rightarrow$ (ii) of \Cref{proposition extension of G} to the maps $h = f|_{\partial V \times [0, \infty)}$, $H = F|_{\partial V \times [0, \infty)}$, and $g = f|_{\partial V}$, $G = F|_{V}$, to obtain a map $v \colon S(G) \rightarrow \mathbb{R}^{2}$ such that $v(x) \notin dG_{x}(T_{x}V) \subset T_{G(x)} \mathbb{R}^{2} = \mathbb{R}^{2}$ for all $x \in S(G)$, and there is $\varepsilon' \in (0, \varepsilon)$ such that $\sigma_{h}(y) \cdot v(x) \in \{0\} \times (0, \infty)$ for all $x = (s, y) \in S(G) \cap ([0, \varepsilon') \times \partial V) = [0, \varepsilon') \times S(g)$.
Next, we apply implication (i) $\Rightarrow$ (ii) of \Cref{proposition cusps and vector fields} to $g = f|_{\partial V}$, $G = F|_{V}$, $\sigma = \sigma_{f}$ ($= \sigma_{h}$), and $v \colon S(G) \rightarrow \mathbb{R}^{2}$, to conclude that $G = F|_{V}$ satisfies condition (ii) of \Cref{proposition cusps and vector fields}.
Finally, we apply implication (i) $\Rightarrow$ (ii) of \Cref{proposition condition n even} to the maps $g = f|_{\partial V}$, $\sigma = \sigma_{f}$ and $G = F|_{V}$, to conclude that $\chi(V) \equiv \chi_{+}(g; \sigma) \; (\operatorname{mod} 2)$.
Using that $\chi(V) \equiv \chi(M) \; (\operatorname{mod} 2)$ by \Cref{lemma euler characteristic}(b), and $\chi_{+}(g; \sigma) = \chi_{+}[f]$ (see \Cref{Creation and elimination of cusps}), we obtain $\chi(M) \equiv \chi_{+}[f] \; (\operatorname{mod} 2)$.
We observe that $\chi(M) = \chi(M_{1}) + \chi(M_{2})$ and $\chi_{+}[f] = \chi_{+}[f_{1}] + \chi_{+}[-f_{2}]$.
Moreover, as
$\# S(f_{2}|_{\partial M_{2}})$ is even by \Cref{corollary generic maps}, we get
$$
\chi_{+}[-f_{2}] \stackrel{(\ref{signed Euler characteristic})}{\equiv} \# S_{+}[-f_{2}|_{\partial M_{2}}] = \# S(f_{2}|_{\partial M_{2}}) - \# S_{+}[f_{2}|_{\partial M_{2}}] \equiv - \chi_{+}[f_{2}] \; (\operatorname{mod} 2).
$$
All in all, it follows that $\chi(M_{1}) - \chi_{+}[f_{1}] \equiv \chi(M_{2}) - \chi_{+}[f_{2}] \; (\operatorname{mod} 2)$ as desired.

Finally, the resulting map $\mathcal{C}_{n} \rightarrow \mathbb{Z}/2$ is clearly additive because the group law on $\mathcal{C}_{n}$ is induced by disjoint union.
\end{proof}

\begin{proposition}\label{proposition existence of homomorphism for n odd}
If $n > 2$ is odd, then there exists a homomorphism
\begin{align*}
\mathcal{C}_{n} \rightarrow \mathbb{Z}, \qquad [f \colon M \rightarrow \mathbb{R}] \mapsto \chi(M) - \chi_{+}[f].
\end{align*}
\end{proposition}

\begin{proof}
Analogously to the proof of \Cref{proposition existence of homomorphism}, we can show that if two Morse functions $f_{1} \colon M_{1} \rightarrow \mathbb{R}$ and $f_{2} \colon M_{2} \rightarrow \mathbb{R}$ defined on compact oriented $n$-manifolds possibly with boundary represent the same element in $\mathcal{C}_{n}$, then the Morse function $f = f_{1} \sqcup -f_{2}$ defined on $M = M_{1} \sqcup -M_{2}$ satisfies $\frac{\chi(\partial M)}{2} = \chi_{+}(f|_{\partial M}; \sigma_{f})$.
(Note that we have to employ \Cref{proposition condition n odd} instead of \Cref{proposition condition n even}.)
Using that $\chi(\partial M) = \chi(\partial M_{1}) + \chi(\partial M_{2})$ and $\chi_{+}(f|_{\partial M}; \sigma_{f}) = \chi_{+}(f_{1}|_{\partial M_{1}}; \sigma_{f_{1}}) + \chi_{+}(-f_{2}|_{\partial M_{2}}; \sigma_{-f_{2}})$, we obtain
$$
\frac{\chi(\partial M_{1})}{2} - \chi_{+}(f_{1}|_{\partial M_{1}}; \sigma_{f_{1}}) = - \frac{\chi(\partial M_{2})}{2} + \chi_{+}(-f_{2}|_{\partial M_{2}}; \sigma_{-f_{2}}).
$$
Furthermore, since $n$ is odd, \Cref{helpful expression} implies that
$$
- \frac{\chi(\partial M_{2})}{2} + \chi_{+}(-f_{2}|_{\partial M_{2}}; \sigma_{-f_{2}}) = \frac{\chi(\partial M_{2})}{2} - \chi_{+}(f_{2}|_{\partial M_{2}}; \sigma_{f_{2}}).
$$
All in all, using that $\frac{\chi(\partial M_{j})}{2} = \chi(M_{j})$ by \Cref{lemma euler characteristic}(a), and $ \chi_{+}(f_{j}|_{\partial M_{j}}; \sigma_{f_{j}}) = \chi_{+}[f_{j}]$ (see \Cref{Creation and elimination of cusps}) for $j = 1, 2$, we obtain $\chi(M_{1}) - \chi_{+}[f_{1}] = \chi(M_{2}) - \chi_{+}[f_{2}]$ as desired.

Finally, the resulting map $\mathcal{C}_{n} \rightarrow \mathbb{Z}$ is clearly additive because the group law on $\mathcal{C}_{n}$ is induced by disjoint union.
\end{proof}

It remains to show that the homomorphisms of \Cref{proposition existence of homomorphism} and \Cref{proposition existence of homomorphism for n odd} are isomorphisms.

It follows from \Cref{remark realization of sign distribution} that the homomorphisms of \Cref{proposition existence of homomorphism} and \Cref{proposition existence of homomorphism for n odd} are surjective.
To see this, we fix an arbitrary compact oriented $n$-manifold $M$ with nonempty boundary, e.g. $M = D^{n}$, and an arbitrary Morse function $g \colon \partial M \rightarrow \mathbb{R}$.
Note that for any map $\sigma \colon S(g) \rightarrow \{\pm 1\}$, $g$ can be extended by means of \Cref{remark realization of sign distribution} to a Morse function $f \colon M \rightarrow \mathbb{R}$ in such a way that $\sigma = \sigma_{f}$.
When $n$ is even, we choose the map $\sigma \colon S(g) \rightarrow \{\pm 1\}$ in such a way that $\# S_{+}(g; \sigma) \not\equiv \chi(M) \; (\operatorname{mod} 2)$.
Since $\# S_{+}(g; \sigma) \equiv \chi_{+}(g; \sigma) \; (\operatorname{mod} 2)$, we have thus achieved that the homomorphism of \Cref{proposition existence of homomorphism} maps the resulting class $[f \colon M \rightarrow \mathbb{R}] \in \mathcal{C}_{n}$ to the generator of $\mathbb{Z}/2$.
When $n$ is odd, we use \Cref{helpful expression} to choose the map $\sigma \colon S(g) \rightarrow \{\pm 1\}$ in such a way that $\frac{\chi(\partial M)}{2} - \chi_{+}(g; \sigma) = 1$.
Since $\frac{\chi(\partial M)}{2} = \chi(M)$ by \Cref{lemma euler characteristic}(a), we have achieved that the homomorphism of \Cref{proposition existence of homomorphism for n odd} maps the resulting class $[f \colon M \rightarrow \mathbb{R}] \in \mathcal{C}_{n}$ to a generator of $\mathbb{Z}$.

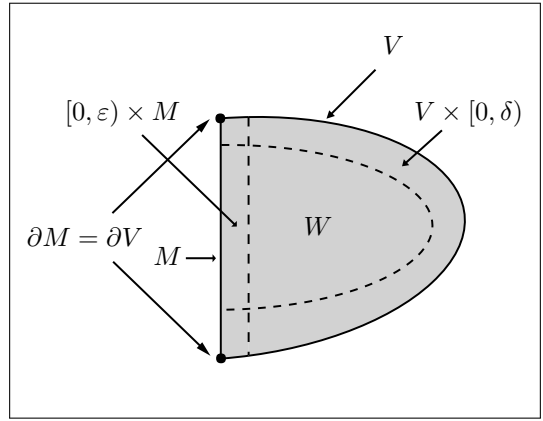
\begin{figure}[htbp]
\centering
\fbox{\begin{tikzpicture}
    \draw (0, 0) node {\includegraphics[height=0.4\textwidth]{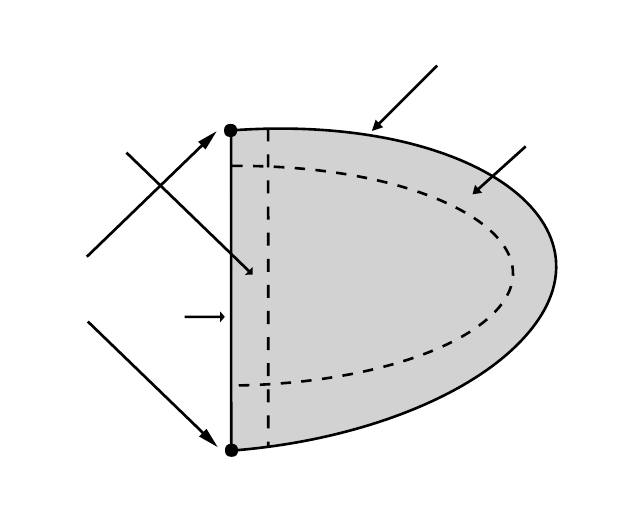}};
    \draw (-2.6, -0.35) node {$\partial M = \partial V$};
    \draw (1.5, 2.2) node {$V$};
    \draw (2.5, 1.3) node {$V \times [0, \delta)$};
    \draw (-1.5, -0.6) node {$M$};
    \draw (-2.1, 1.3) node {$[0, \varepsilon) \times M$};
    \draw (0.5, -0.2) node {$W$};
\end{tikzpicture}}
\caption{Illustration of an oriented cobordism $(W, V)$ from $M$ to $\emptyset$.
Note that $W$ is a compact oriented manifold with boundary $\partial W = M \cup V$ and corners along $\partial M = \partial V$.
The relevant collar neighborhoods (with corners) of $M \subset W$ and $V \subset W$ are also indicated.}
\label{figure injectivity}
\end{figure}

It remains to show that the homomorphisms of \Cref{proposition existence of homomorphism} and \Cref{proposition existence of homomorphism for n odd} are injective.
For this purpose, we suppose that $f \colon M \rightarrow \mathbb{R}$ represents an element $[f] \in \mathcal{C}_{n}$ which has trivial image under the homomorphisms of \Cref{proposition existence of homomorphism} and \Cref{proposition existence of homomorphism for n odd} for $n \geq 2$ even and odd, respectively.
That is, we assume that $\chi(M) \equiv \chi_{+}[f] \; (\operatorname{mod} 2)$ when $n$ is even, and $\chi(M) = \chi_{+}[f]$ when $n$ is odd.
As indicated in \Cref{figure injectivity}, we choose a compact oriented $(n+1)$-manifold with corners $W^{n+1}$ such that $\partial W = M \cup_{\partial M} V$, where $V^{n}$ is a compact oriented $n$-manifold with boundary $\partial V = \partial M$, and the corners of $W$ are along $\partial V$.
(In fact, we can take $W = M \times [0, 1]$.
Note that by smoothing the corners along $\partial M \times \{1\}$, we can achieve that $V = (\partial M \times [0, 1]) \cup_{\partial M \times \{1\}} (M \times \{1\})$ becomes a smooth manifold.)
Without loss of generality we may assume that $V$ is connected by applying the boundary connected sum operation to the (non-compact) manifold $W \setminus M$ with boundary $V \setminus \partial V$.
Writing $g = f|_{\partial M}$ and $\sigma = \sigma_{f}$, we note that $\chi_{+}[f] = \chi_{+}(g; \sigma)$ in the notation of \Cref{Creation and elimination of cusps} (see \Cref{abstract sign Euler characteristic}).
Moreover, by \Cref{lemma euler characteristic}, we have $\chi(M) \equiv \chi(V) \; (\operatorname{mod} 2)$ for $n$ even, and $\chi(M) = \frac{\chi(V)}{2}$ for $n$ odd.
Hence, by the assumptions on $f$, we may apply implication (ii) $\Rightarrow$ (i) of \Cref{proposition condition n even} when $n$ is even, and implication (ii) $\Rightarrow$ (i) of \Cref{proposition condition n odd} when $n$ is odd, to obtain a map $G \colon V \rightarrow \mathbb{R}^{2}$ such that $G|_{V \setminus \partial V}$ is generic, $G|_{[0, \varepsilon) \times \partial V} = \operatorname{id}_{[0, \varepsilon)} \times g$ in some collar neighborhood $[0, \varepsilon) \times \partial V \subset V$ of $\partial V \subset V$, and $G$ satisfies condition (ii) of \Cref{proposition cusps and vector fields}.
Then, applying implication (ii) $\Rightarrow$ (i) of \Cref{proposition cusps and vector fields} to $g = f|_{\partial V}$, $G$, and $\sigma = \sigma_{f}$, we obtain a map $v \colon S(G) \rightarrow \mathbb{R}^{2}$ such that $v(x) \notin dG_{x}(T_{x}V) \subset T_{G(x)} \mathbb{R}^{2} = \mathbb{R}^{2}$ for all $x \in S(G)$, and there is $\varepsilon' \in (0, \varepsilon)$ such that $\sigma(y) \cdot v(x) \in \{0\} \times (0, \infty)$ for all $x = (s, y) \in S(G) \cap ([0, \varepsilon') \times \partial V) = [0, \varepsilon') \times S(g)$.
We choose a collar neighborhood (with corners) $V \times [0, \infty) \subset W$ of $V \subset W$ (see \Cref{figure injectivity}).
Then, we may apply implication (ii) $\Rightarrow$ (i) of \Cref{proposition extension of G} to the maps $h = f|_{\partial V \times [0, \infty)}$, $g = f|_{\partial V}$, $G$, and $v \colon S(G) \rightarrow \mathbb{R}^{2}$ (where note that $\sigma = \sigma_{f} = \sigma_{h}$) to obtain a map $H \colon V \times [0, \infty) \rightarrow \mathbb{R}^{2}$ such that $H|_{V \times \{0\}} = G$, $H|_{[0, \varepsilon') \times \partial V \times [0, \infty)} = \operatorname{id}_{[0, \varepsilon')} \times h$ for some $\varepsilon' \in (0, \varepsilon)$, and $H|_{(V \setminus \partial V) \times [0, \infty)}$ is admissible in the sense of \Cref{Non-singular extensions}.
We can choose a collar neighborhood (with corners) $[0, \varepsilon) \times M \subset W$ of $M \subset W$ (see \Cref{figure injectivity}) that extends $[0, \varepsilon) \times \partial V \times [0, \delta) \subset V \times [0, \infty)$ for some $\delta > 0$.
Finally, we extend the $M \cup_{\partial M} V$-germ of the map
$$
H|_{V \times [0, \delta)} \cup (\operatorname{id}_{[0, \varepsilon')} \times f) \colon (V \times [0, \delta)) \cup ([0, \varepsilon') \times M) \rightarrow \mathbb{R}^{2}
$$
by means of \Cref{proposition generic extension} to the desired cusp cobordism $F \colon W \rightarrow \mathbb{R}^{2}$ from $f$ to $f_{\emptyset}$ (compare \Cref{definition cusp cobordism} and \Cref{remark target plane}).

This completes the proof of \Cref{main theorem}.

\section{Admissible cobordism group}\label{admissible cobordism group}
Recall from \Cref{Non-singular extensions} that a map between manifolds possibly with boundary is called \emph{admissible} if it is a submersion on a neighborhood of the boundary of the source manifold.
As noted before, Saeki and Yamamoto's notion of admissible cobordism \cite{sy1, sy3} (see \Cref{definition admissible cobordism} below) is slightly different from our notion of cusp cobordism (see \Cref{definition cusp cobordism}) in that they make additional $C^{\infty}$ stability assumptions on the maps.
The purpose of this section is to show that both cobordism relations yield isomorphic cobordism groups (see \Cref{Admissible cobordism}).
In preparation, we collect some necessary background about $C^{\infty}$ stable maps in \Cref{stable maps}.

\subsection{Stable maps}\label{stable maps}
We recall that a map $f \colon N \rightarrow P$ of a manifold possibly with boundary $N$ to a manifold without boundary $P$ is called \emph{$C^{\infty}$ stable} if there exists a neighborhood of $f$ in the space $C^{\infty}(N, P)$ of maps from $N$ to $P$ endowed with the Whitney $C^{\infty}$ topology such that every map $g$ in the neighborhood is $C^{\infty}$ right-left equivalent to $f$, that is, there exist diffeomorphisms $\Phi \colon N \rightarrow N$ and $\Psi \colon P \rightarrow P$ such that $\Psi \circ f = g \circ \Phi$.

For instance, when $N$ is a compact manifold possibly with boundary, a function $f \colon N \rightarrow \mathbb{R}$ is $C^{\infty}$ stable if and only if $f$ is a Morse function that is injective on $S(f) \sqcup S(f|{\partial N})$.
Furthermore, in the case that $N$ (not necessarily compact) has dimension $\geq 3$ and $P$ is a surface without boundary, it is well-known that a proper admissible map $f \colon N \rightarrow P$ is $C^{\infty}$ stable if and only if both $f|_{N \setminus \partial N}$ and $f|_{\partial N}$ are generic maps, and the restriction of $f$ to the $1$-manifold without boundary $S(f|_{N \setminus \partial N}) \sqcup S(f|_{\partial N})$ has the following properties.
If $\Sigma \subset S(f|_{N \setminus \partial N}) \sqcup S(f|_{\partial N})$ denotes the set of cusps of $f|_{N \setminus \partial N}$ and $f|_{\partial N}$, then $f|_{(S(f|_{N \setminus \partial N}) \sqcup S(f|_{\partial N})) \setminus \Sigma}$ is an immersion with normal crossings (see e.g. Section III.\S 3 in \cite[pp. 82ff]{gol}), and for every $c \in \Sigma$ we have $f^{-1}(f(c)) = \{c\}$.

In view of our application to admissible cobordism, we shall exploit the following result, which is Lemma 3.3.6 in \cite[p. 62]{wra}.

\begin{proposition}\label{perturbation of fold lines}
Let $n \geq 2$ be an integer.
Consider the normal form
\begin{align}\label{fold normal form}
\Lambda(t, z_{1}, \dots, z_{n-1}) = (t, -z_{1}^{2} - \dots - z_{i}^{2} + z_{i+1}^{2} + \dots + z_{n-1}^{2})
\end{align}
for a fold point of absolute index $i$ (see \Cref{definition generic maps into the plane}).
Given functions $\alpha, \beta \colon \mathbb{R} \rightarrow \mathbb{R}$ with compact support, we set
\begin{align*}
\Delta(t, z_{1}, \dots, z_{n-1}) = (0, \alpha(t) \beta(||z||^{2})),
\end{align*}
where $||z||^{2} = z_{1}^{2} + \dots + z_{n}^{2}$.
If $|\alpha(t) \beta'(r)| < 1$ for all $(t, r) \in \mathbb{R}^{2}$, then the perturbation $F = \Lambda + \Delta$ of $\Lambda$ is a fold map with a single fold line $S(F) = S(\Lambda) = \mathbb{R} \times \{0\} \subset \mathbb{R}^{n}$ whose image $F(S(F)) \subset \mathbb{R}^{2}$ is given by the graph of the map $t \mapsto (t, \alpha(t)\beta(0))$.
\end{proposition}

\begin{proof}
The proof that $F$ is a fold map is a straightforward application of the standard fact (see Proposition 3.3.4(c) in \cite[p. 58]{wra}) that a map $H \colon \mathbb{R}^{n} = \mathbb{R} \times \mathbb{R}^{n-1} \rightarrow \mathbb{R}^{2}$ of the form $H(t, z) = (t, h(t, z))$ for some function $h \colon \mathbb{R} \times \mathbb{R}^{n-1} \rightarrow \mathbb{R}^{2}$ is a fold map if and only if $0 \in \mathbb{R}^{n-1}$ is a regular value of $D^{z}h = (\partial_{z_{1}}h, \dots, \partial_{z_{n-1}}h) \colon \mathbb{R} \times \mathbb{R}^{n-1} \rightarrow \mathbb{R}^{n-1}$, and the restriction $H|_{S(H)} \colon S(H) \rightarrow \mathbb{R}^{2}$ of $H$ to the $1$-dimensional submanifold $S(H) = (D^{z}h)^{-1}(0) \subset \mathbb{R}^{n}$ is an immersion.
The remaining claims follow immediately.
\end{proof}

Given a generic map $G \colon X^{n} \rightarrow \mathbb{R}^{2}$ defined on a manifold (without boundary) of dimension $n \geq 2$, \Cref{corollary perturb} below enables us to perturb the image of a fold line of $G$ near a given fold point in a controlled way such that the singular set $S(G)$ is not changed by the perturbation.

\begin{corollary}\label{corollary perturb}
Let $n \geq 2$ be an integer.
Let $U \subset \mathbb{R}^{n} = \mathbb{R} \times \mathbb{R}^{n-1}$ and $V \subset \mathbb{R}^{2}$ be open subsets such that $\Lambda(U) \subset V$, where $\Lambda$ denotes the normal form (\ref{fold normal form}) for fold points.
Suppose that $[a, b] \times \{0\} \subset U$ for some real numbers $a < b$.
Then, there exists $c > 0$ such that for any function $\gamma \colon \mathbb{R} \rightarrow [-c, c]$ with support in $[a, b]$, there exists a fold map $G \colon U \rightarrow V$ such that $G|_{U \setminus K} = \Lambda|_{U \setminus K}$ for some compact subset $K \subset U$, and such that $G$ has a single fold line $S(G) = S(\Lambda|_{U})$ whose image $G(S(G)) \subset \mathbb{R}^{2}$ is given by the graph of the map $t \mapsto (t, \gamma(t))$ defined for $t \in \mathbb{R}$ with $(t, 0) \in U$.
\end{corollary}

\begin{proof}
The construction of $c > 0$ is as follows.
As $\Lambda([a, b] \times \{0\}) = [a, b] \times \{0\} \subset V$, we may choose $\varepsilon > 0$ such that $[a, b] \times [-2\varepsilon, 2\varepsilon] \subset V$.
Then, there exists $\delta > 0$ such that $K = [a, b] \times \{z \in \mathbb{R}^{n-1}| \; ||z||^{2} \leq \delta\} \subset U \cap \Lambda_{2}^{-1}((-\varepsilon, \varepsilon))$, where we write $\Lambda = (\Lambda_{1}, \Lambda_{2})$.
Let $\beta \colon \mathbb{R} \rightarrow \mathbb{R}$ be a function with support in $[-\delta, \delta]$ such that $\beta(0) = 1$.
We choose $c > 0$ so small that $|\beta(r)| < \varepsilon/c$ and $|\beta'(r)| < 1/c$ for all $r \in \mathbb{R}$.

Given a function $\gamma \colon \mathbb{R} \rightarrow [-c, c]$ with support in $[a, b]$, we construct the desired fold map $G \colon U \rightarrow V$ as follows.
Since $|\gamma(t)\beta'(r)| \leq c|\beta'(r)| < 1$ for all $(t, r) \in \mathbb{R}^{2}$ by choice of $c$, we may apply \Cref{perturbation of fold lines} to the functions $\alpha = \gamma$ and $\beta$ to obtain a fold map $F = \Lambda + \Delta$ with a single fold line $S(F) = S(\Lambda) = \mathbb{R} \times \{0\} \subset \mathbb{R}^{n}$ whose image $F(S(F)) \subset \mathbb{R}^{2}$ is given by the graph of the map $t \mapsto (t, \alpha(t)\beta(0)) = (t, \gamma(t))$.
Recall that the map $\Delta \colon \mathbb{R}^{n} = \mathbb{R} \times \mathbb{R}^{n-1} \rightarrow \mathbb{R}^{2}$ is of the form $\Delta = (0, \Delta_{2})$, and the map $(t, r) \mapsto \Delta_{2}(t, r) = \alpha(t) \beta(||z||^{2})$ vanishes outside of $K = [a, b] \times \{z \in \mathbb{R}^{n-1}| \; ||z||^{2} \leq \delta\} \subset U$.
It remains to show that $F(U) \subset V$, from which it follows immediately that $G = F| \colon U \rightarrow V$ has the desired properties.
Let $(t, z) \in U$.
If $(t, z) \notin K$, then it follows from $\Delta(t, z) = (0, 0)$ that $F(t, z) = \Lambda(t, z) \in V$.
Otherwise, we have $(t, z) \in K = [a, b] \times \{z \in \mathbb{R}^{n-1}| \; ||z||^{2} \leq \delta\}$.
Then, we conclude from $|\Delta_{2}(t, z)| = |\alpha(t) \beta(||z||^{2})| \leq c |\beta(||z||^{2})| < \varepsilon$ and $|\Lambda_{2}(t, z)| \leq \varepsilon$ that $F(t, z) = (t, (\Lambda_{2} + \Delta_{2})(t, z)) \in [a, b] \times [-2\varepsilon, 2\varepsilon] \subset V$.
\end{proof}

\subsection{Admissible cobordism}\label{Admissible cobordism}
The notion of (oriented) admissible cobordism (see Definition 1.1 as well as Section 6 in \cite{sy3}) is defined as follows.

\begin{definition}[admissible cobordism]\label{definition admissible cobordism}
For $i = 0, 1$ let $f_{i} \colon M_{i} \rightarrow \mathbb{R}$ be a $C^{\infty}$ stable function on a compact oriented $n$-manifold $M_{i}$ possibly with boundary.
An \emph{(oriented) admissible cobordism} from $f_{0}$ to $f_{1}$ is an (oriented) cobordism $(W^{n+1}, V)$ from $M_{0}$ to $M_{1}$ together with a map $F \colon W \rightarrow [0, 1] \times \mathbb{R}$ such that $F^{-1}(\mathbb{R} \times \{i\}) = M_{i}$ for $i = 0, 1$, and the following properties hold:
\begin{enumerate}[(i)]
\item For some $\varepsilon > 0$ there exist collar neighborhoods (with corners) $[0, \varepsilon) \times M_{0} \subset W$ of $M_{0} \subset W$ and $(1-\varepsilon, 1] \times M_{1} \subset W$ of $M_{1} \subset W$ such that $F|_{[0, \varepsilon) \times M_{0}} = \operatorname{id}_{[0, \varepsilon)} \times f_{0}$ and $F|_{(1-\varepsilon, 1] \times M_{1}} = \operatorname{id}_{(1-\varepsilon, 1]} \times f_{1}$.
\item The restriction $F| \colon W \setminus (M_{0} \sqcup M_{1}) \rightarrow (0, 1) \times \mathbb{R}$ is a proper admissible $C^{\infty}$ stable map.
\end{enumerate}
\end{definition}
Following \cite{sy3}, the resulting oriented and unoriented admissible cobordism groups of admissible Morse functions on manifolds with boundary are denoted by $b \mathfrak{M}_{n}$ and $b \mathfrak{N}_{n}$, respectively.

Since $C^{\infty}$ stable functions on compact manifolds possibly with boundary are Morse functions, and admissible cobordisms are cusp cobordisms in the sense of \Cref{definition cusp cobordism}, we have a well-defined homomorphism
\begin{align}\label{homomorphism from admissible cobordism to cusp cobordism}
b \mathfrak{M}_{n} \rightarrow \mathcal{C}_{n}, \qquad [f \colon M^{n} \rightarrow \mathbb{R}] \mapsto [f].
\end{align}
In response to Problem 6.2 in \cite{sy3}, we prove the following

\begin{proposition}\label{proposition admissible cobordism group}
The homomorphism (\ref{homomorphism from admissible cobordism to cusp cobordism}) is an isomorphism.
\end{proposition}

\begin{proof}
In order to show that the homomorphism (\ref{homomorphism from admissible cobordism to cusp cobordism}) is surjective, we consider a Morse function $f \colon M \rightarrow \mathbb{R}$ representing a class in $\mathcal{C}_{n}$.
By perturbing the Morse function $f|_{\partial M} \colon \partial M \rightarrow \mathbb{R}$ slightly, we obtain a Morse function $g \colon \partial M \rightarrow \mathbb{R}$ which is injective on $S(g)$.
We apply \Cref{remark realization of sign distribution} to the Morse function $g \colon \partial M \rightarrow \mathbb{R}$ and the map $\sigma \colon S(g) \rightarrow \{\pm 1\}$ induced by $\sigma_{f} \colon S(f|_{\partial M}) \rightarrow \{\pm 1\}$ to obtain a Morse function $f_{1} \colon M \rightarrow \mathbb{R}$ such that $f_{1}|_{\partial M} = g$ and $\chi_{+}[f_{1}] = \chi_{+}[f]$.
Moreover, by perturbing the critical points of $f_{1}$, we may assume that $f_{1}$ is injective on $S(f_{1}) \sqcup S(f_{1}|_{\partial M})$.
Thus, $f_{1}$ is a $C^{\infty}$ stable function, and \Cref{main theorem} implies that $f_{1}$ represents the class $[f \colon M \rightarrow \mathbb{R}] \in \mathcal{C}_{n}$.

Conversely, as for the proof that the homomorphism (\ref{homomorphism from admissible cobordism to cusp cobordism}) is injective, we suppose that the Morse function $f \colon M \rightarrow \mathbb{R}$ represents a class in $b \mathfrak{M}_{n}$, and is nullcobordant in $\mathcal{C}_{n}$.
Then, by \Cref{definition cusp cobordism}, there is an oriented cobordism $(W, V)$ from $M$ to $\emptyset$, and a map $F_{0} \colon W \rightarrow [0, 1] \times \mathbb{R}$ such that $F_{0}^{-1}(\mathbb{R} \times \{0\}) = M$, $F_{0}^{-1}(\mathbb{R} \times \{1\}) = \emptyset$, and the following properties hold:
\begin{enumerate}[(i)]
\item For some $\varepsilon > 0$ there exists a collar neighborhood (with corners) $[0, \varepsilon) \times M \subset W$ of $\{0\} \times M = M \subset W$ such that $F_{0}|_{[0, \varepsilon) \times M} = \operatorname{id}_{[0, \varepsilon)} \times f$.
\item The restriction $F_{0}|_{W \setminus M}$ is an admissible map.
\item The restrictions $F_{0}|_{W \setminus \partial W}$ and $F_{0}|_{V \setminus \partial V}$ are generic maps into the plane.
\end{enumerate}
We choose a collar neighborhood (with corners) $V \times [0, \infty) \subset W$ of $V \subset W$ that extends $[0, \varepsilon') \times \partial V \times [0, \infty) \subset [0, \varepsilon') \times M$ for some $\varepsilon' \in (0, \varepsilon)$.
Applying implication (i) $\Rightarrow$ (ii) of \Cref{proposition extension of G} to the maps $h = f|_{\partial V \times [0, \infty)}$ and $G = F_{0}|_{V}$, we obtain a map $v \colon S(F_{0}|_{V}) \rightarrow \mathbb{R}^{2}$ with the properties of \Cref{proposition extension of G}(ii) with respect to $h = f|_{\partial V \times [0, \infty)}$ and $G = F_{0}|_{V}$.
Hence, it follows from the implication (i) $\Rightarrow$ (ii) of \Cref{proposition cusps and vector fields} applied to $g = f|_{\partial V}$, $G = F_{0}|_{V}$, and $\sigma = \sigma_{f}$, that $G = F_{0}|_{V}$ satisfies property (ii) of \Cref{proposition cusps and vector fields} with respect to $g = f|_{\partial V}$ and $\sigma = \sigma_{f}$.
Since $f|_{\partial M}$ is a $C^{\infty}$ stable function, we may perturb the generic map $F_{0}|_{V \setminus \partial V} \colon V \setminus \partial V \rightarrow \mathbb{R}^{2}$ slightly on a compact subset of $V \setminus \partial V$ to obtain a map $G_{0} \colon V \rightarrow \mathbb{R}^{2}$ such that the restriction $G_{0}|_{V \setminus \partial V} \colon V \setminus \partial V \rightarrow \mathbb{R}^{2}$ is a proper $C^{\infty}$ stable map.
(The desired perturbation can be obtained by means of standard techniques as follows.
First, we let the cusps of $H_{0} = G_{0}|_{V \setminus \partial V}$ propagate slightly as explained in Lemma (3) in Section (4.6) in \cite[p. 290]{lev} to achieve that no two cusps of the resulting generic map $H_{1} \colon V \setminus \partial V \rightarrow \mathbb{R}^{2}$ have the same image point in the plane.
Finally, by means of \Cref{corollary perturb}, we perturb $H_{1}$ near its fold lines on a compact subset of $V \setminus \partial V$ in order to produce the desired proper $C^{\infty}$ stable map $H_{2} = G_{0}|_{V \setminus \partial V} \colon V \setminus \partial V \rightarrow \mathbb{R}^{2}$.
More precisely, we perturb $H_{1}$ near a finite number of small embedded compact intervals $[0, 1] \rightarrow S(H_{1}) \setminus \Sigma$, where $\Sigma$ denotes the set of cusps of $H_{1}$.
These intervals are chosen to be pairwise disjoint, and such that $H_{1}$ restricts to an immersion with normal crossings on the complement of their union in $S(H_{1}) \setminus \Sigma$.)
After the perturbation, we have $G_{0}|_{[0, \varepsilon'') \times \partial V} = \operatorname{id}_{[0, \varepsilon'')} \times f|_{\partial V}$ for some $\varepsilon'' \in (0, \varepsilon')$, and $G = G_{0}$ still satisfies property (ii) of \Cref{proposition cusps and vector fields} with respect to $g = f|_{\partial V}$ and $\sigma = \sigma_{f}$.
Hence, it follows from the implication (ii) $\Rightarrow$ (i) of \Cref{proposition cusps and vector fields} applied to $g = f|_{\partial V}$, $G = G_{0}$, and $\sigma = \sigma_{f}$, that there is a map $v \colon S(G_{0}) \rightarrow \mathbb{R}^{2}$ with the properties of \Cref{proposition cusps and vector fields}(i) with respect to $g = f|_{\partial V}$, $G = G_{0}$, and $\sigma = \sigma_{f}$.
By the implication (ii) $\Rightarrow$ (i) of \Cref{proposition extension of G} applied to $h = f|_{\partial V \times [0, \infty)}$ and $G = G_{0}$, there exists a map $H \colon V \times [0, \infty) \rightarrow \mathbb{R}^{2}$ such that $H|_{V \times \{0\}} = G_{0}$, $H|_{[0, \varepsilon''') \times \partial V \times [0, \infty)} = \operatorname{id}_{[0, \varepsilon''')} \times f|_{\partial V \times [0, \infty)}$ for some $\varepsilon''' \in (0, \varepsilon'')$, and $H|_{(V \setminus \partial V) \times [0, \infty)}$ is admissible.
Next, we use \Cref{proposition generic extension} to extend the $M \cup_{\partial M} V$-germ of
$$
H|_{V \times [0, \infty)} \cup (\operatorname{id}_{[0, \varepsilon''')} \times f) \colon (V \times [0, \infty)) \cup ([0, \varepsilon''') \times M) \rightarrow \mathbb{R}^{2}
$$
to a cusp cobordism $F_{1} \colon W \rightarrow \mathbb{R}^{2}$ from $f$ to $f_{\emptyset} \colon \emptyset \rightarrow \mathbb{R}$ in the sense of \Cref{definition cusp cobordism} (compare \Cref{remark target plane}).
Finally, by perturbing $F_{1}$ slightly on a compact subset of $W \setminus \partial W$ by means of the same techniques that we used in the construction of $G_{0} \colon V \rightarrow \mathbb{R}^{2}$ above, we obtain a cusp cobordism $F_{2} \colon W \rightarrow \mathbb{R}^{2}$ from $f$ to $f_{\emptyset} \colon \emptyset \rightarrow \mathbb{R}$ such that $F_{2}|_{V} = G_{0}$, and the restriction $F_{2}|_{W \setminus M}$ is a proper admissible $C^{\infty}$ stable map.
Thus, $F_{2}$ is the desired admissible cobordism from $f$ to $f_{\emptyset} \colon \emptyset \rightarrow \mathbb{R}$.
\end{proof}

\bibliographystyle{amsplain}

\end{document}